\documentclass[leqno]{amsart}
\usepackage{amssymb}
\usepackage{palatino}
\usepackage{dsfont}
\usepackage{cases}
\usepackage{color}
\usepackage[T1]{fontenc}

\newtheorem{theorem}{Theorem}[section]
\newtheorem{proposition}[theorem]{Proposition}
\newtheorem{lemma}[theorem]{Lemma}

\theoremstyle{definition}
\newtheorem{definition}[theorem]{Definition}

\theoremstyle{remark}
\newtheorem{remark}[theorem]{Remark}
 
\numberwithin{equation}{section}

\DeclareMathOperator{\dist}{dist}
\DeclareMathOperator{\Div}{div}
\DeclareMathOperator*{\Limsup}{\limsup{}^{*}}
\DeclareMathOperator*{\Liminf}{\liminf{}_{*}}

\begin{document}
\title[Second order PDEs with Caputo time fractional derivatives]{On existence and uniqueness of viscosity solutions for second order fully nonlinear PDEs with Caputo time fractional derivatives}

\author{Tokinaga~Namba}
\address{Graduate School of Mathematical Sciences, The University of Tokyo 3-8-1 Komaba, Meguro-ku, Tokyo, 153-8914 Japan}
\email{namba@ms.u-tokyo.ac.jp}

\begin{abstract}
Initial-boundary value problems for second order fully nonlinear PDEs with Caputo time fractional derivatives of order less than one are considered in the framework of viscosity solution theory.
Associated boundary conditions are Dirichlet and Neumann, and they are considered in the strong sense and the viscosity sense, respectively.
By a comparison principle and Perron's method, unique existence for the Cauchy-Dirichlet and Cauchy-Neumann problems are proved.
\end{abstract}



\maketitle
\section{Introduction}
Let $T>0$ and $\alpha\in(0,1)$ be given constants and $\Omega$ be a bounded domain in $\mathbf{R}^d$.
We are concerned with existence and uniqueness for viscosity solutions of initial-boundary value problems for
\begin{equation}\label{e:main}
\partial_t^\alpha u+F(t,x,u,\nabla u,\nabla^2 u)=0\quad\text{in $(0,T]\times\Omega$.}
\end{equation}
Here $\nabla u$ and $\nabla^2 u$ denote the spatial gradient and the Hessian matrix of the unknown function $u:[0,T]\times\Omega\to\mathbf{R}$, respectively.
Moreover, $\partial_t^\alpha u$ denotes the Caputo (time) fractional derivative \cite{Caputo2008} defined as
\begin{equation}\label{e:caputo}
	(\partial_t^\alpha u)(t,x)=
	\begin{cases}
		\displaystyle\frac{1}{\Gamma(1-\alpha)}\int_0^t\frac{\partial_tu(s,x)}{(t-s)^\alpha}ds\quad&\text{for $\alpha\in(0,1)$,}\\
		(\partial_tu)(t,x)\quad&\text{for $\alpha=1$,}
	\end{cases}
\end{equation}
where $\Gamma$ is the usual gamma function.
A given real-valued function $F$ is continuous and degenerate elliptic (see \eqref{e:degenerate} for the definition).
We note that by the degenerate ellipticity, \eqref{e:main} includes first order cases $F=F(t,x,w,p)$.

Differential equations with fractional derivatives have attracted great interest from both mathematics and applications within the last few decades. Among fractional derivatives, the Caputo fractional derivative is effectively used for modeling many real phenomena (\cite{AllenCaffarelliVasseur2016}, \cite{AllenCaffarelliVasseur2017}, \cite{AnnaFominChugunovNiiboriHashida2016}, \cite{Caputo2008}, \cite{FoufoulaGeorgiouGanti2010}, \cite{KlagesRadonsSokolov}, \cite{KolokoltsovVeretennikova2014}, \cite{SchumerMeerschaertBaeumer2009}, and \cite{Voller2014}). A feature of the Caputo derivative is the memory and trapping effect.
Many studies have been done also for the Riemann-Liouville fractional derivative, which is defined by exchanging the order of differentiation and integration in the definition of the Caputo fractional derivative.
However, in order to solve differential equations with Riemann-Liouville fractional derivatives, data for a fractional integral of the unknown function and their integer order derivatives should be given on the endpoints.
For Caputo fractional derivatives, one can consider in the same setting as in the integer order case (see \cite{Diethelm2010} for example). Therefore, the Caputo fractional derivative is better for modeling phenomena.
There are enormous attempts to obtain exact solutions and develop numerical schemes so far (see \cite{BaleanuDiethelmScalasTrujillo2017}, \cite{DiethelmFordFreedLuchko2005}, \cite{LiZheng2015}, and references therein).
For further topics on fractional calculus we refer the reader to \cite{Diethelm2010}, \cite{KilbasSrivastavaTrujilo2006}, \cite{Podlubny1999}, \cite{SamkoKilbasMarichev1993}, and \cite{Zhou2014}.

As in the case $\alpha=1$, a smooth solution is not always expected to exist, so it is necessary to consider a solution in a weak sense.
Various notions of solutions have been introduced.
In \cite{Pruss} Pr\"{u}ss introduces notions of a strong solution, a mild solution, and a weak solution of linear evolutionary Volterra equations
\begin{equation}\label{e:volterra}
u(t)=f(t)+(k\ast Au)\quad\text{in $[0,T]$}
\end{equation}
for a continuous function $u:[0,T]\to\mathbf{R}$.
Here $A$ is a closed linear unbounded operator in a complex Banach space $X$ with dense domain $\mathcal{D}(A)$, $k\in L_{loc}^1([0,\infty))$ is a scalar kernel, $f\in C([0,\infty);X)$, and $(k\ast g)(t)=\int_0^t k(t-s)g(s)ds$.
Utilizing these notions, he shows that there exists a unique strong solution that satisfies a certain property if and only if \eqref{e:volterra} admits a ``resolvent''.
He also constructs a resolvent using spectrum theory in a special case.
We note that PDEs (partial differential equations) of the form
\begin{equation}\label{e:linear}
\partial_t^\alpha u+Au=0
\end{equation}
are expressed as \eqref{e:volterra} with $f=u(\cdot,x)$ and $k(t)=t^{-\alpha}$ by applying the Riemann-Liouville integral
$$
I^{1-\alpha}u(t)=\frac{1}{\Gamma(1-\alpha)}\int_0^t\frac{u(s)}{(t-s)^\alpha}ds
$$
to both sides of \eqref{e:linear}.
There are similar considerations for nonlinear equations.
Hern\'{a}ndez, O'Regan, and Balachandran \cite{HernandezOReganBalachandran2010} considers abstract evolution equations of the form
$$
\partial_t^\alpha(u+g(\cdot,u))(t)=Au+f(t,u(t)).
$$
Here $A$ is an infinitesimal generator of a $C^0$-semigroup of bounded linear operators on Banach space $X$, and $f,g\in C([0,T]\times X,\mathcal{D}(A))$. They define a mild solution according to \cite{Pruss} and establish a unique existence result using the fixed point theorem under appropriate conditions for $A,f$ and $g$.
Kolokoltsov and Veretennikova \cite{KolokoltsovVeretennikova2014} considers 
\begin{equation}\label{e:kolokoltsov}
\partial_t^\alpha u=(-\Delta)^{\beta/2}u+H(t,x,\nabla u)\quad\text{in $(0,\infty)\times\mathbf{R}^d$}
\end{equation}
with a Lipschitz continuous $H$ and a fractional Laplacian 
$$
-(-\Delta)^{\beta/2}u(t,x)=C_{d,\beta}\ p.v.\int_{\mathbf{R}^d}\frac{u(t,x)-u(t,y)}{|x-y|^{d+\beta}}dy
$$
of $\beta\in(1,2]$.
Here $C_ {d,\beta}$ is a normalizing constant and $p.v.$ stands for principal value. 
Their definition of a mild solution is based on an integral equation (mild form) derived by the Fourier transformation unlike \cite{Pruss}.
In fact, a fractional ordinary differential equation derived by taking the Fourier transformation for \eqref{e:kolokoltsov} in space can be solved. They define a mild solution with an integral equation obtained by taking the inverse Fourier transformation for the solved one in space.
They prove that for an initial data belonging to $C^1_\infty(\mathbf{R}^d)$ there exists a unique mild solution belonging to $C([0,T],C^1_\infty(\mathbf{R}^d)$. Here $C^1_\infty(\mathbf{R}^d)$ is the set of functions $f\in C^1(\mathbf{R}^d)$ such that $f$ and $\nabla f$ are continuous functions decreasing rapidly. The argument is based on the fixed point theorem. Their assumptions on $H$ are actually a bit weaker than Lipschitz continuity.
They also prove that the mild solution is a classical solution under some restrictions for $H$ and initial data.

Luchko \cite{Luchko2010} considers an initial-boundary value problem for \eqref{e:linear} with $Au=-\Div(p(x)u)+q(x)u$ and continuous initial-boundary data. Here $p\in C^1$ is a uniformly positive function and $q$ is a continuous nonnegative function. He defines a generalized solution for a continuous function as the uniform convergence limit of a classical solution of \eqref{e:linear} associated with initial-boundary data uniformly converging to given ones.
He constructs a formal solution using the Fourier method of the separation (the eigenfunction expansion) and proves that it is a generalized solution and it is also a classical solution under certain restrictions.
The uniqueness is guaranteed by a maximum principle for an initial-boundary value problem established in \cite{Luchko2009}.
For the purpose of discussing an inverse problem, Sakamoto and Yamamoto \cite{SakamotoYamamoto2011} introduces a notion of a weak solution of an initial-boundary value problem for \eqref{e:linear} with $Au=-\Div(p(x)u)+q(x)u$ including a continuous exterior term $f(t,x)$.
For a function belonging to a certain class, their weak solution is defined as the equation holds in $L^2$ sense.
They prove a unique existence of the weak solution by using the eigenfunction expansion (Galerkin method) and a priori estimate. They also give asymptotic estimates as time goes to infinity.
Zacher \cite{Zacher2009} finds a unique solution satisfying
$$
\frac{d}{dt}(k\ast (u-x)(t),v)_{\mathcal{H}}+a(t,u(t),v)=\langle f(t),v\rangle_{\mathcal{V}'\times\mathcal{V}}\quad\text{$v\in\mathcal{V}$, a.a. $t\in(0,T)$}
$$
in a class $\{u\in L^2([0,T],\mathcal{V}) \mid k\ast(u-x)\in {}_{0}H^1_2([0,T],\mathcal{V}')\}$ for given $x\in\mathcal{H}$ and $f\in L^2([0,T],\mathcal{V}')$. Here $\mathcal{V}$ and $\mathcal{H}$ are real separable Hilbert spaces such that $\mathcal{V}$ is densely and continuously embedded into $\mathcal{H}$, $k\in L^1_{loc}([0,\infty))$ is a scalar kernel that belongs to a certain class, and $A:(0,T)\times\mathcal{V}\times\mathcal{V}\to\mathbf{R}$ is a bounded $\mathcal{V}$-coercive bilinear form. Moreover, $(\cdot,\cdot)_\mathcal{H}$ and $\langle\cdot,\cdot\rangle_{\mathcal{V}',\mathcal{V}}$ denote the scalar product in $\mathcal{H}$ and the duality pairing between $\mathcal{V}$ and its dual $\mathcal{V}'$, respectively. The zero of ${}_{0}H^1_2([0,T],\mathcal{V}')$ means vanishing trace at $t=0$. The argument is based on the Galerkin method and a priori estimate.
We note that $\partial_t^\alpha u(t)=\frac{d}{dt}I^{1-\alpha}[u-u(0)](t)$ holds almost all $t\in(0,T]$ if $u=u(t)$ is an absolute continuous function (see \cite[Lemma 2.12]{Diethelm2010}), and that $k\ast(u-x)(t)=I^{1-\alpha}[u-x](t)$ when $k(t)=t^{-\alpha}$.
As notions of solutions coming from completely different ideas, there are those using a form obtained by certain integration by parts formula (\cite{AllenCaffarelliVasseur2016} and \cite{AllenCaffarelliVasseur2017}) and those based on defining the Caputo fractional derivative on a finite interval of a certain fractional Sobolev space (\cite{GorenfloLuchkoYamamoto2015}).

Taking into account of generalizations of nonlinearity $F$ and boundary condition, we use a viscosity solution theory.
A viscosity solution is a kind of generalized solution for PDEs and was introduced by Crandall and Lions \cite{CrandallLions1983}.
We refer the reader to \cite{BardiCapuzzoDolcetta1997} and \cite{Koike2004} for basic theory and to \cite{CannarsaSinestrari2004}, \cite{CrandallIshiiLions1992}, and \cite{Giga2006} for more advanced theory.
While the viscosity solution theory develops greatly for local equations, it has also developed for nonlocal equations 
$$
F (x,u,\nabla u,\nabla^2 u,\mathcal{I}[u])=0
$$
with much general nonlocal terms $\mathcal{I}[u]$ due to contributions of many researchers.
In \cite{BarlesImbert2008} by Barles and Imbert, there is a set of results for the well-posedness in the framework of viscosity solution theory, including the story so far.

A viscosity solution of equations with Caputo time fractional derivatives was first given by Allen \cite{Allen}. He considers a regularity issue for viscosity solutions of nonlocal nondivergence nonlinear parabolic equation of the type
$$
\partial_t^\alpha u-\sup_i\inf_j\int_{\mathbf{R}^d}\frac{u(t,x+y)-u(t,x)}{|y|^{d+2\beta}}a^{ij}(t,x,y)dy=0.
$$
Here $a^{ij}$ is a uniformly positive function that satisfies $a^{ij}(t,x,y)=a^{ij}(t,x,-y)$.
The Caputo fractional derivative is formally rewritten as
\begin{equation}\label{e:marchaud}
(\partial_t^\alpha u)(t)=\frac{\alpha}{\Gamma(1-\alpha)}\int_{-\infty}^t\frac{u(t)-\tilde{u}(s)}{(t-s)^{\alpha+1}}ds=:M_t[u],
\end{equation}
where $\tilde{u}$ is a function such that $\tilde{u}(t)=u(t)$ for $t\ge0$ and $\tilde{u}(t)=u(0)$ for $t<0$.
The form of $M_t$, which is an analogue of the Marchaud fractional derivative (see, e.g., \cite[Definition 2.11]{KlagesRadonsSokolov} and \cite[Section 5.4]{SamkoKilbasMarichev1993}), is very close to the fractional Laplacian. Therefore he handles $M_t$ in accordance with the idea of the above-mentioned nonlocal theory of viscosity solutions and defines a viscosity solution.

Actually, $M_t$ is included in the class of nonlocal terms handled in \cite{BarlesImbert2008}. Thus, even if the Caputo time fractional derivative is added to PDEs that could be considered so far, the well-posedness would be verified by regarding the time variable as a part of the space variable. 
Most of the methods taken in this paper are in fact standard in the viscosity solution theory. 
However, there are several differences between spatial nonlocal operators such as fractional Laplacian and the Caputo time fractional derivative.
For example, fractional Laplacian is defined in the whole space $\mathbf{R}^d$ in principle. Hence, for boundary value problems, boundary conditions need to be imposed not only on the boundary but also on the outside (See \cite{JacobKnopova2005} for example). On the other hand, the Caputo fractional derivative always makes sense in a bounded interval, so initial conditions can be imposed in the same way as usual case. We also point out that the Caputo fractional derivative is non-translation invariant.
Mou and \'{S}wi\k{e}ch \cite{MouSwiech2015} considers a uniqueness issue for viscosity solutions of Bellman-Isaacs type equation with a nonlocal term
$$
\mathcal{I}_x[u]=\int_{\mathbf{R}^d}[u(x+z)-u(x)-\mathds{1}_{B(0,1)}(z)\nabla u(z)\cdot z]\mu_x(dz).
$$
Here $\mathds{1}_{B(0,1)} $ is the indicator function of unit ball $B(0,1)$ and $\{\mu_x\}_x$ is a family of L\'{e}vy measures. Notice that $\mathcal{I}_x$ is not translation invariant in general due to the $x$ dependency of $\mu_x$. They prove a comparison principle after getting an appropriate regularity for viscosity sub- and supersolutions, since a standard proof does not work because of the non-translation invariance (see also \cite{Mou2017}). In our situation, a comparison principle can be proved for an upper semicontinuous subsolution and a lower semicontinuous supersolution in the standard way.
In addition, as described later in Introduction, a difference between roles of time derivative and space derivative is used especially in our proof of a comparison principle.
Therefore, it is expected that handling the Caputo time fractional derivative independently will give beneficial observations.
Motivated by some of these, Giga and the author \cite{GigaNamba2017} consider the initial value problem for the Hamilton-Jacobi equations 
$$
\partial_t^\alpha u+H(t,x,u,\nabla u)=0\quad\text{in $(0,T]\times\mathbf{T}^d$}
$$
under the spatial periodic boundary condition.
They prove a unique existence, stability, and a regularity of the solution under standard assumptions on $H$ and initial data. We note that their viscosity solution is essentially the same as one by Allen's idea.

The purpose of this paper is to develop the arguments by \cite{GigaNamba2017} and to investigate a unique existence for viscosity solutions of initial-boundary value problems for \eqref{e:main}. 
Specifically, we consider the Cauchy-Dirichlet problem with the homogeneous Dirichlet condition and the Cauchy-Neumann problem with the homogeneous Neumann condition. In either case, the homogeneity is chosen for simplicity, not essential.
As is well known, the viscosity solution theory allows us to handle very general boundary conditions \cite{CrandallIshiiLions1992}. 
However, it is inappropriate to always interpret boundary conditions in ``the strong sense" that they are imposed at each point on the boundary, and it is required to interpret in ``the viscosity sense''.
The homogeneous Dirichlet and Neumann conditions are chosen as typical examples in the strong and viscosity sense, respectively.

By using the maximum principle as usual, derivatives of the unknown function in \eqref{e:main} are replaced by those of a test function, so we can define a viscosity solution of \eqref{e:main}. Here the maximum principle by Lucho \cite{Luchko2009} is used for the Caputo fractional derivative. Our definition of a viscosity solution (Definition \ref{d:solution}) is rewritten using $M_t$. To be precise, a nonlocal operator which slightly rewritten $M_t$ are handled in this paper for convenience sake. In \cite{GigaNamba2017} the solution by this definition was called a provisional solution. The well-posedness was not clarified as it is difficult to prove a comparison principle. We prove that our definition is equivalent to that where $M_t[\varphi]$ for a test function $\varphi$ is replaced with $M_t[u]$ for the unknown function $u$ (Proposition \ref{p:alternative1}). At the same time, it is guaranteed that $M_t[u]$ exists at points where the unknown function is tested. Similar facts are stated by Arisawa \cite{Arisawa2008}, Barles and Imbert \cite{BarlesImbert2008}, and Jakobsen and Karlsen \cite{JakobsenKarlsen2006} for nonlocal equations without Caputo time fractional derivatives.
This equivalent definition is an extension of one by \cite{Allen} and \cite{GigaNamba2017} to second order equations.

We prove an existence theorem by Perron's method developed by Ishii \cite{Ishii1987}. A proof is similar to that of \cite{GigaNamba2017}, but by virtue of the equivalence of definitions of solutions, some of the proofs can be simplified. Precisely, $M_t[\varphi]$ is continuous with respect to the independent variables, and so the standard proof works.

The uniqueness is ensured by the comparison principle. Our proof is slightly different and simpler than that of $\alpha=1$, so let us explain briefly here. To that end, let $u$ and $v$ be smooth solutions of \eqref{e:main} that satisfy $u\le v$ on $(\{0\}\times\overline{\Omega})\cup([0,T]\times\partial\Omega)$. Suppose by contradiction that $\max_{[0,T]\times\overline{\Omega}}(u-v)=(u-v)(t,x)>0$ at some point $(t,x)\in(0,T]\times\Omega$. Let us assume that $F(t,x,w,p,X)=F(p)$ for the presentation simplicity.
We know that
$$
\partial_t^\alpha u(t,x)+F(\nabla u(t,x))\le0\quad\text{and}\quad \partial_t^\alpha v(t,x)+F(\nabla v(t,x))\ge0.
$$
Since $F(\nabla u(t,x))=F(\nabla v(t,x))$, subtracting both sides yields
\begin{equation}\label{e:inequ}
\partial_t^\alpha(u-v)(t,x)\le0.
\end{equation}
When $\alpha=1$, there is no contradiction from \eqref{e:inequ}. Hence $-\eta t$ with a sufficiently small constant $\eta>0$ is often added to $u-v$. Then contradiction is obtained since it follows that $0<\eta=\eta+\partial_t^1(u-v)\le 0$. On the other hand, in the case of $\alpha\in(0,1)$, it also holds that
$$
M_t[u-v](t,x)\le 0.
$$
Notice that
\begin{align*}
&M_t[u-v](t,x)\\
&=\frac{(u-v)(t,x)-(u-v)(0,x)}{t^\alpha\Gamma(1-\alpha)}+\frac{\alpha}{\Gamma(1-\alpha)}\int_0^t\frac{(u-v)(t,x)-(u-v)(s,x)}{(t-s)^\alpha}ds.
\end{align*}
Since $(u-v)(t,x)\ge (u-v)(s,x)$ for all $s\in[0,T]$, it follows that
$$
\frac{(u-v)(t,x)-(u-v)(0,x)}{t^\alpha\Gamma(1-\alpha)}\le0.
$$
From the supposition by the contradiction and the assumption of the comparison principle, it turns out that the left-hand side is positive, a contradiction.

In the proof of the comparison principle for second order equations in the case of $\alpha=1$, ``the maximum principle for semicontinuous functions'' established by Crandall and Ishii \cite{CrandallIshii1990} is effectively used. Unfortunately, it is difficult to use it directly for nonlocal equations. This is explained carefully by Jakobsen and Karlsen in \cite[Section 2]{JakobsenKarlsen2006}. Very roughly speaking, it is difficult to find an appropriate integrable function that bounds the integrand which guarantees the use of some convergence theorem from an element in the closure of a set of semijets.
In \cite{JakobsenKarlsen2006} they establish a ``nonlocal version'' of \cite[Proposition 5.1]{Ishii1989} which is the same kind of result as the maximum principle for semicontinuous functions, and prove a comparison principle for equations with nonlocal terms. The assumption is weakened by Barles and Imbert \cite{BarlesImbert2008}.
We present a similar result (Lemma \ref {l:ishii}) of \cite[Lemma 7.4]{JakobsenKarlsen2006} and \cite[Proposition IV.1]{IshiiLions1990} that can be applied to equations with Caputo time fractional derivatives with the same idea. For the proof we do not use results established in \cite{JakobsenKarlsen2006} and \cite{BarlesImbert2008}.
A notion of viscosity solution when boundary conditions are interpreted in the viscosity sense is introduced in the same way as in the case of $\alpha=1$. By employing techniques used so far, a unique existence theorem of a viscosity solution for the Cauchy-Neumann problem is proved without trouble.

In this paper, we only consider the homogeneous Dirichlet and Neumann conditions. Some extensions to other boundary conditions would be possible. For variations of boundary conditions the reader is referred to \cite{CrandallIshiiLions1992} and references therein. See also \cite{BarlesIshiiMitake2012} and \cite{BarlesMitake2012} for first order equations.
We note that a viscosity solution of \eqref{e:main} with the state constraint boundary condition also can be defined in a similar manner and the well-posedness is discussed. In the case of periodic boundary conditions viscosity sub- and supersolutions required by Perron's method can be easily constructed following \cite[Corollary 4.3]{GigaNamba2017}. 

In order to consider more complicated phenomena, generalization of the definition of the Caputo fractional derivative in several directions is attempted. One of them is the distributed order Caputo fractional derivative, which is defined as
$$
(\mathbb{D}_t^{(\omega)}u)(t)=\int_0^1(\partial_t^\alpha u)(t)\omega(\alpha)d\alpha.
$$
Here $\omega$ is a nonnegative weight function. Mathematical analysis of equations with this derivative is not yet much, but important observations on linear equations have been made. Li, Luchko, and Yamamoto \cite{LiLuchkoYamamoto2014}, \cite{LiLuchkoYamamoto2017} show an existence and uniqueness of solution by using the eigenfunction expansion and Laplace transformation.
They also show the short-and long-time behavior \cite{LiLuchkoYamamoto2014} and the analyticity in time of solution \cite{LiLuchkoYamamoto2017}, and mention references on application.
The results of this paper hold even for \eqref{e:main} where $\partial_t^\alpha$ is replaced by $\sum_{i=1}^n\lambda_i\partial_t^{\alpha_i}$. Here $\alpha_i\in(0,1]$ and $\lambda_i>0$. When $\omega(\alpha)=\sum_{i=1}^n\lambda_i\delta(\alpha=\alpha_i)$, $\mathbb{D}_t^{(\omega)}$ is represented as above, where $\delta$ is a Dirac delta function. This special case is also considered for linear equations (\cite{Luchko2011} and \cite{LiLiuYamamoto2015} for example).
There seems to be some research when $\lambda_i$ depends on variables like $\lambda_i=\lambda_i(t,x)$ as another generalization. 
It is also interesting to discuss the well-posedness of \eqref{e:main} where $\partial_t^\alpha$ is replaced with $\sum_{i=1}^n\lambda_i(t,x)\partial_t^{\alpha_i}$ or $\mathbb{D}_t^{(\omega)}$.
However, for the former case, our arguments in the proof of the comparison principle does not work due to technical reasons, and so the well-posedness remains open. For the latter case, the definition of a solution is not clear.

This paper is organized as follows: In Section 2 we first summarize properties of the nonlocal operator $M_t$.  We then give a definition of a viscosity solution of \eqref{e:main}, equivalent definitions, and stability. In Section 3 we prove a unique existence theorem for the Cauchy-Dirichlet problem with homogeneous Dirichlet boundary condition. The key lemma in the proof of the comparison principle is proved in Section 4. In Section 5 we give a definition of a viscosity solution when boundary conditions are interpreted in the viscosity sense, and then prove a uniqueness existence theorem for the Cauchy-Neumann problem with a homogeneous Neumann boundary condition.

\section{Definition of a viscosity solution and its basic properties}

We first summarize properties of a nonlocal operator.
For a function $f\in C^1((0,T])\cap C([0,T])$ such that $f'\in L^1(0,T)$, the Caputo fractional derivative $(\partial_t^\alpha f)(t)$ is defined for every $t\in(0,T]$.
Moreover, the integration by parts and the change of variable of integration imply that
$$
(\partial_t^\alpha f)(t)=J[f](t)+K_{(0,t)}[f](t)\quad\text{for $t\in(0,T]$,}
$$
where
\begin{align*}
&J[f](t):=\frac{f(t)-f(0)}{t^\alpha\Gamma(1-\alpha)}\quad\text{and}\\
&K_{(0,t)}[f](t):=\frac{\alpha}{\Gamma(1-\alpha)}\int_0^t (f(t)-f(t-\tau))\frac{d\tau}{\tau^{\alpha+1}}.
\end{align*}
The subscript $(0,t)$ of $K_{(0,t)}$ denotes the interval of integration in the definition of $K_{(0,t)}[f](t)$. 
By natural extension, we define $K_{(a,b)}[f](t)$ for any $a,b$ with $0\le a<b\le t$. 
For a function $u:[0,T]\times A\to\mathbf{R}$, where $A$ is a set in $\mathbf{R}^N$ with $N\ge 1$, we write $J[u](t,x):=J[u(\cdot,x)](t)$ and $K_{(a,b)}[u](t,x):=K_{(a,b)}[u(\cdot,x)](t)$.

Let $t\in(0,T]$ and $a\in[0,t)$.
In this paper $K_{(a,t)}[f](t)$ for semicontinuous functions $f:[0,T]\to\mathbf{R}$ is often handled.
When $a>0$, $K_{(a,t)}[f](t)$ is interpreted in the sense of Lebesgue integral. When $a=0$, $K_{(0,t)}[f](t)$ is regarded as $\lim_{a\searrow0}K_{(a,t)}[f](t)$, and we say that $K_{(0,t)}[f](t)$ \emph{exists} if $K_{(a,t)}[f^\pm](t)$ are finite for each $a\in(0,t)$ and $\lim_{a\searrow0}K_{(a,t)}[f^\pm](t)$ exists as a finite number.
Here $f^\pm:=\max\{\pm f,0\}$.

\begin{proposition}\label{p:propertyK}
Let $f:[0,T]\to\mathbf{R}$ be an upper (resp. lower) semicontinuous function.
Then $K_{(a,t)}[f](t)$ is bounded from below (resp. above) for each $a$ and $t$ with $0<a<t\le T$.
If $f\in C^1((a,T])\cap C([0,T])$ with $0\le a<T$, then $K_{(0,t)}[f](t)$ exists for each $t\in(a,T]$ and it is continuous with respect to $t$ in $(a,T]$.
\end{proposition}

\begin{proof}
The first assertion is due to the extreme value theorem for semicontinuous functions.

Let $t\in(0,T]$ and $\rho>0$ be such that $2\rho<\min\{t-a,T-t\}$.
For $s\in[t-\rho,t+\rho]$ and $\tau\in(0,T)$, we have
\begin{align*}
&\frac{|f(s)-f(s-\tau)|}{\tau^{\alpha+1}}\mathds{1}_{(0,s)}(\tau)\\
&=\frac{|f(s)-f(s-\tau)|}{\tau^{\alpha+1}}\mathds{1}_{(0,\rho)}(\tau)+\frac{|f(s)-f(s-\tau)|}{\tau^{\alpha+1}}\mathds{1}_{(\rho,s)}(\tau)\\
&\le\frac{\max_{[t-\rho,t+\rho]\cap(0,T]}|f'|}{\tau^\alpha}\mathds{1}_{(0,\rho)}(\tau)+\frac{2\max_{[0,T]}|f|}{\tau^{\alpha+1}}\mathds{1}_{(\rho,T)}(\tau).
\end{align*}
Here $\mathds{1}_{I}$ is the indicator function on an interval $I$, that is, $\mathds{1}_{I}(\tau)=1$ for $\tau\in I$ and $0$ for $\tau\not\in I$.
The last function is integrable on $(0,T)$.
Therefore the dominated convergence theorem ensures the second and third assertions.
\end{proof}

We next define a viscosity solution of
\begin{equation}\label{e:master2}
\partial_t^\alpha u+F(t,x,u,\nabla u,\nabla^2 u)=0\quad\text{in $(a,T]\times O$,}
\end{equation}
where $0\le a<T$, $O$ is a open set in $\mathbf{R}^d$, and $F$ is a real-valued function on $W:=(a,T]\times O\times\mathbf{R}\times\mathbf{R}^d\times\mathbf{S}^d$.
Here $\mathbf{S}^d$ denotes the space of real $d\times d$ symmetric matrices with the usual ordering ``$\le$'', that is, $X\le Y$ if $\langle X\xi,\xi\rangle\le\langle Y\xi,\xi\rangle$ for all $\xi\in\mathbf{R}^d$.
Remember that $\partial_t^\alpha$ is defined as \eqref{e:caputo}

Let $A$ be a set in $\mathbf{R}^N$ with $N\ge 1$.
For a function $h:A\to\mathbf{R}$ let $h^*$ and $h_*$ denote the upper and lower semicontinuous envelope, respectively. Namely,
$$
h^*(x)=\lim_{\sigma\searrow0}\sup\{h(y)\mid y\in A\cap\overline{B(x,\sigma)}\}
$$
and $h_*(x)=-(-h)^*(x)$ for $x\in A$.
Here $B(x,r)$ is an open ball of radius $r$ centered at $x$ and $\overline{B(x,r)}$ is its closure.
Throughout this section, we always assume that $-\infty<F_*\le F^*<+\infty$ in $W$. 

\begin{definition}[Viscosity solution]\label{d:solution}
(i) A function $u:[0,T]\times O\to\mathbf{R}$ is a viscosity subsolution of \eqref{e:master2} in $(a,T]\times  O$ if 
$u^*<+\infty$ on $[0,T]\times O$ and
\begin{equation}\label{e:subsolution}
J[\varphi](\hat{t},\hat{x})+K_{(0,\hat{t})}[\varphi](\hat{t},\hat{x})
+F_*(\hat{t},\hat{x},u^*(\hat{t},\hat{x}),\nabla\varphi(\hat{t},\hat{x}),\nabla^2\varphi(\hat{t},\hat{x}))\le0
\end{equation}
whenever $((\hat{t},\hat{x}),\varphi)\in((a,T]\times O)\times (C^{1,2}((a,T]\times O)\cap C([0,T]\times O))$ satisfies
$$
\max_{[0,T]\times O}(u^*-\varphi)=(u^*-\varphi)(\hat{t},\hat{x}).
$$

(ii) A function $u:[0,T]\times O\to\mathbf{R}$ is a viscosity supersolution of \eqref{e:master2} in $(a,T]\times  O$ if 
$u_*>-\infty$ in $[0,T]\times O$ and
$$
J[\varphi](\hat{t},\hat{x})+K_{(0,\hat{t})}[\varphi](\hat{t},\hat{x})+F^*(\hat{t},\hat{x},u_*(\hat{t},\hat{x}),\nabla\varphi(\hat{t},\hat{x}),\nabla^2\varphi(\hat{t},\hat{x}))\ge0
$$
whenever $((\hat{t},\hat{x}),\varphi)\in((a,T]\times O)\times (C^{1,2}((a,T]\times O)\cap C([0,T]\times O))$ satisfies
$$
\min_{[0,T]\times O}(u_*-\varphi)=(u_*-\varphi)(\hat{t},\hat{x}).
$$

(iii) If a function $u:[0,T]\times O\to\mathbf{R}$ is both a viscosity sub- and supersolution of \eqref{e:master2} in $(a,T]\times O$, then $u$ is called a viscosity solution of \eqref{e:master2} in $(a,T]\times O$.
\end{definition}

Here, by $C^{1,2}((a,T]\times O)$, we mean the space of functions $\varphi$ such that $\varphi,\partial_t\varphi, \nabla\varphi$ and $\nabla^2\varphi$ are continuous in $(a,T)\times  O$ and some neighborhood of $\{T\}\times O$.
We hereafter suppress the word ``viscosity'' unless confusion occurs.

\begin{remark}
When $F$ is continuous and degenerate elliptic in the sense that
\begin{equation}\label{e:degenerate}
\begin{split}
&\text{for all $(t,x,w,p)\in(0,T]\times O\times\mathbf{R}\times\mathbf{R}^d$,}\\
&F(t,x,w,p,Y)\le F(t,x,w,p,X)\quad\text{if $X\le Y$,}
\end{split}
\end{equation}
the notion of a viscosity solution by Definition \ref{d:solution} is consistent with that of a classical solution.
More precisely, $u\in C^{1,2}((0,T]\times O)\cap C([0,T]\times O)$ such that $\partial_tu(\cdot,x)\in L^1(0,T)$ for every $x\in O$ is a viscosity solution of \eqref{e:master2} in $(0,T]\times O$ if and only if it is a classical solution that satisfies \eqref{e:master2} pointwise in $(0,T]\times O$.
\end{remark}

\begin{remark}
As mentioned in the Introduction, all assertions in this paper hold for a multi-term case
\begin{equation}\label{e:multi}
\sum_{i=0}^n\lambda_i\partial_t^{\alpha_i}u+F(t,x,u,\nabla u,\nabla^2u)=0,
\end{equation}
where $\alpha_i\in(0,1]$ and $\lambda_i>0$ are constants.
A subsolution $u:[0,T]\times O\to\mathbf{R}$ of \eqref{e:multi} in $(a,T]\times O$ is defined by replacing $J[\varphi](\hat{t},\hat{x})+K_{(0,\hat{t})}[\varphi](\hat{t},\hat{x})$ in \eqref{e:subsolution} with $\sum_{i=1}^n\lambda_i(J^{\alpha_i}[\varphi](\hat{t},\hat{x})+K_{(0,\hat{t})}^{\alpha_i}[\varphi](\hat{t},\hat{x}))$. Here $J^{\alpha_i}$ and $K^{\alpha_i}$ are operators $J$ and $K$ associated with $\alpha=\alpha_i$, respectively. A supersolution and a solution are defined in a similar way.
\end{remark}

Definition \ref{d:solution} is suitable for proving an existence theorem by Perron's method.
For a comparison principle it is convenient to introduce equivalent definitions.
\begin{proposition}\label{p:alternative1}
A function $u:[0,T]\times O\to\mathbf{R}$ is a subsolution (resp. supersolution) of \eqref{e:master2} in $(a,T]\times O$ if and only if $u^*<+\infty$ (resp. $u_*>-\infty$) in $[0,T]\times O$, $K_{(0,\hat{t})}[u^*](\hat{t},\hat{x})$ (resp. $K_{(0,\hat{t})}[u_*](\hat{t},\hat{x})$) exists, and
\begin{align*}
&J[u^*](\hat{t},\hat{x})+K_{(0,\hat{t})}[u^*](\hat{t},\hat{x})+F_*(\hat{t},\hat{x},u^*(\hat{t},\hat{x}),\nabla\varphi(\hat{t},\hat{x}),\nabla^2\varphi(\hat{t},\hat{x}))\le0\\
\text{(resp., }&J[u_*](\hat{t},\hat{x})+K_{(0,\hat{t})}[u_*](\hat{t},\hat{x})+F^*(\hat{t},\hat{x},u_*(\hat{t},\hat{x}),\nabla\varphi(\hat{t},\hat{x}),\nabla^2\varphi(\hat{t},\hat{x}))\ge0\text{)}
\end{align*}
whenever $((\hat{t},\hat{x}),\varphi)\in((a,T]\times O)\times C^{1,2}((a,T]\times O)$ satisfies
$$
\max_{(a,T]\times O}(u^*-\varphi)=(u^*-\varphi)(\hat{t},\hat{x})\quad(\text{resp. }\min_{(a,T]\times O}(u^*-\varphi)=(u^*-\varphi)(\hat{t},\hat{x})).
$$
\end{proposition}

\begin{proof}
We prove only for subsolution since the proof for supersolution is symmetric.
The `if' part is easy.
Indeed, the assumption $\max_{[0,T]\times O}(u^*-\varphi)=(u^*-\varphi)(\hat{t},\hat{x})$ implies that $\varphi(\hat{t},\hat{x})-\varphi(\hat{t}-\tau,\hat{x})\le u^*(\hat{t},\hat{x})-u^*(\hat{t}-\tau,\hat{x})$ for all $\tau\in[0,\hat{t}]$.
Thus we have 
$$
J[\varphi](\hat{t},\hat{x})+K_{(0,\hat{t})}[\varphi](\hat{t},\hat{x})\le J[u^*](\hat{t},\hat{x})+K_{(0,\hat{t})}[u^*](\hat{t},\hat{x}),
$$
which immediately yields the desired inequality.

In order to prove the `only if' part we take $((\hat{t},\hat{x}),\varphi)\in ((a,T]\times O)\times C^{1,2}((a,T]\times O)$ that satisfies $\max_{(a,T]\times O}(u^*-\varphi)=(u^*-\varphi)(\hat{t},\hat{x})$.
We may assume that $(u^*-\varphi)(\hat{t},\hat{x})=0$ by adding $(u^*-\varphi)(\hat{t},\hat{x})$ to $\varphi$.

Fix $\rho>0$ such that $\hat{t}-\rho>a$.
Since $u^*$ is usc (upper semicontinuous), there exists a sequence $u_\sigma\in C([0,T]\times O)$ such that $u_\sigma\searrow u^*$ pointwise in $[0,T]\times O$ as $\sigma\searrow0$.
Moreover, there exists a sequence $\varphi_\sigma\in C^{1,2}((0,T]\times O)\cap C([0,T]\times O)$ such that $\varphi_\sigma=\varphi$ in $\overline{B((\hat{t},\hat{x}),\rho/2)}\cap((0,T]\cap O)$, $u^*\le \varphi_\sigma\le u_\sigma+\sigma$ in $([0,T]\times O)\setminus B((\hat{t},\hat{x}),\rho)$, and $u^*\le\varphi_\sigma\le\varphi$ in $B((\hat{t},\hat{x}),\rho)$ (see, e.g., \cite{Arisawa2008} for such a construction of $\varphi_\sigma$). It is easy to see that $\max_{[0,T]\times O}(u^*-\varphi_\sigma)=(u^*-\varphi_\sigma)(\hat{t},\hat{x})$.

Since $u$ is a subsolution of \eqref{e:master2} in $(a,T]\times O$, we have
\begin{equation}\label{e:equivalent_1}
J[\varphi_\sigma](\hat{t},\hat{x})+K_{(0,\hat{t})}[\varphi_\sigma](\hat{t},\hat{x})+F_*(\hat{t},\hat{x},u^*(\hat{t},\hat{x}),\nabla\varphi_\sigma(\hat{t},\hat{x}),\nabla^2\varphi_\sigma(\hat{t},\hat{x}))\le0
\end{equation}
It is clear that $\lim_{\sigma\to0}J[\varphi_\sigma](\hat{t},\hat{x})=J[u^*](\hat{t},\hat{x})$ and $(\nabla\varphi_\sigma,\nabla^2\varphi_\sigma)=(\nabla\varphi,\nabla^2\varphi)$ at $(\hat{t},\hat{x})$.
We see that
\begin{equation}\label{e:alternative_1.5}
\begin{split}
K_{(0,\hat{t})}[\varphi_\sigma](\hat{t},\hat{x})
&=K_{(0,\rho)}[\varphi_\sigma](\hat{t},\hat{x})+K_{(\rho,\hat{t})}[\varphi_\sigma](\hat{t},\hat{x})\\
&\ge K_{(0,\rho)}[\varphi](\hat{t},\hat{x})+K_{(\rho,\hat{t})}[\varphi_\sigma](\hat{t},\hat{x})
\end{split}
\end{equation}
and that  $\lim_{\sigma\to0}K_{(\rho,\hat{t})}[\varphi_\sigma](\hat{t},\hat{x})=K_{(\rho,\hat{t})}[u^*](\hat{t},\hat{x})$ due to the dominated convergence theorem.
Thus taking the limit infimum in \eqref{e:equivalent_1} as $\sigma\to0$ after estimating by \eqref{e:alternative_1.5}  yields
\begin{equation}\label{e:equivalent_2}
\begin{split}
&J[u^*](\hat{t},\hat{x})+K_{(0,\rho)}[\varphi](\hat{t},\hat{x})+K_{(\rho,\hat{t})}[u^*](\hat{t},\hat{x})\\
&+F_*(\hat{t},\hat{x},u^*(\hat{t},\hat{x}),\nabla\varphi(\hat{t},\hat{x}),\nabla^2\varphi(\hat{t},\hat{x}))\le0.
\end{split}
\end{equation}
Note that this holds for any small $\rho>0$. In order to end the proof, we shall prove that $K_{(0,\hat{t})}[u^*](\hat{t},\hat{x})$ exists and that
\begin{equation}\label{e:equivalent_3}
\lim_{\rho\to0}(K_{(0,\rho)}[\varphi](\hat{t},\hat{x})+K_{(\rho,\hat{t})}[u^*](\hat{t},\hat{x}))=K_{(0,\hat{t})}[u^*](\hat{t},\hat{x}).
\end{equation}

We introduce the function
\begin{equation*}
	v_\rho(\tau)=\begin{cases}
		u^*(\hat{t},\hat{x})-u^*(\hat{t}-\tau,\hat{x})=:v_0(\tau)\quad&\text{for $\tau\in[\rho,\hat{t}]$,}\\
		\varphi(\hat{t},\hat{x})-\varphi(\hat{t}-\tau,\hat{x})\quad&\text{for $\tau\in[0,\rho)$}
	\end{cases}
\end{equation*}
and set
$$
K[v_\rho]:=\frac{\alpha}{\Gamma(1-\alpha)}\int_0^{\hat{t}}v_\rho(\tau)\frac{d\tau}{\tau^{\alpha+1}}=K_{(0,\rho)}[\varphi](\hat{t},\hat{x})+K_{(\rho,\hat{t})}[u^*](\hat{t},\hat{x}).
$$
Notice that $K[v_\rho^-]$ exists for each $\rho>0$ (see Proposition \ref{p:propertyK}).
Fix a small $\rho'>0$.
From \eqref{e:equivalent_2} we have
\begin{equation}\label{e:equivalent_4}
-\infty<K[v_\rho^+]-K[v_{\rho'}^-]+C\le K[v_\rho^+]- K[v_\rho^-]+C\le0\quad\text{for $\rho\le\rho'$},
\end{equation}
where $C:=J[u^*](\hat{t},\hat{x})+F_*(\hat{t},\hat{x},u^*(\hat{t},\hat{x}),\nabla\varphi(\hat{t},\hat{x}),\nabla^2\varphi(\hat{t},\hat{x}))$.
Hence $K[v_\rho^+]$ also exists for each $\rho>0$.
Clearly, $v_\rho^+\nearrow v_0^+$ and $v_\rho^-\searrow v_0^-$ as $\rho\searrow0$.
Thus the monotone convergence theorem implies that $\lim_{\rho\to0}K[v_\rho^+]=K[v_0^+]$ and $\lim_{\rho\to0}K[v_\rho^-]=K[v_0^-]$. 
In view of \eqref{e:equivalent_4} it turns out that $K[v_0^+]$ and $K[v_0^-]$ are finite.
Therefore $K_{(0,\hat{t})}[u^*](\hat{t},\hat{x})=K[v_0^+]+K[v_0^-]$ exists and \eqref{e:equivalent_3} follows.
\end{proof}

\begin{remark}
In (i) of Definition \ref{d:solution} the maximum may be replaced by a strict maximum in the sense that
\begin{equation}\label{e:replacement}
(u^*-\varphi)(t,x)<(u^*-\varphi)(\hat{t},\hat{x}),\quad(t,x)\neq(\hat{t},\hat{x}),\quad (t,x)\in(0,T]\times O
\end{equation}
or by a local strict maximum in the sense that \eqref{e:replacement} holds with some neighborhood of $(\hat{t},\hat{x})$ instead of $(0,T]\times O$.
Similarly, in (ii) of Definition \ref{d:solution}, the minimum may be replaced by a strict minimum or by a strict local minimum.
Similar things are valid for Proposition \ref{p:alternative1}.
\end{remark}

Let $u:[0,T]\times O\to\mathbf{R}$ be such that $-\infty<u_*\le u^*<+\infty$ in $[0,T]\times O$, and $(t,x)\in(a,T]\times O$.
The set of parabolic superjets of $u^*$ and parabolic subjets of $u_*$ at $(t,x)$ are defined as 
\begin{align*}
&\mathcal{P}^+u^*(t,x) \\
&= \left\{(\partial_t\varphi(t,x),\nabla\varphi(t,x),\nabla^2\varphi(t,x)) \middle|
\begin{array}{l}
\text{$\varphi\in C^{1,2}$ such that $u^*-\varphi$ attains}\\
\text{a local maximum at $(t,x)$}
\end{array}\right\}.
\end{align*}
and $\mathcal{P}^-u_*(t,x)=-\mathcal{P}^+(-u)^*(t,x)$, respectively.
We do not use the first derivatives in time, i.e., the first elements of $\mathcal{P}^+ u^*(t,x)$ and $\mathcal{P}^- u_*(t,x)$. Thus we introduce subsets as
$$
\widetilde{\mathcal{P}}^+u^*(t,x):=\{(p,X)\in\mathbf{R}^d\times\mathbf{S}^d \mid \text{there is $a\in\mathbf{R}$ such that $(a,p,X)\in\mathcal{P}^+u^*(t,x)$}\}
$$
and $\widetilde{\mathcal{P}}^-u_*(t,x):=-\widetilde{\mathcal{P}}^+(-u)^*(t,x)$.
Note that if $u$ is once differentiable in time and continuously twice differentiable in space at $(t,x)$, then $\widetilde{\mathcal{P}}^+u^*(t,x)=\widetilde{\mathcal{P}}^-u_*(t,x)=\{(\nabla u(t,x),\nabla^2 u(t,x))\}$.
In view of definitions of these sets, the following equivalence is immediate.

\begin{proposition}\label{p:equivalence}
A function $u:[0,T]\times O\to\mathbf{R}$ is a subsolution (resp. supersolution) of \eqref{e:master2} in $(a,T]\times O$ if and only if $u^*<+\infty$ (resp. $u_*>-\infty$) in $[0,T]\times O$, $K_{(0,t)}[u^*](t,x)$ (resp. $K_{(0,t)}[u_*](t,x)$) exists and
\begin{align*}
&J[u^*](t,x)+K_{(0,t)}[u^*](t,x)+F_*(t,x,u^*(t,x),p,X)\le0\\
\text{(resp. }&J[u_*](t,x)+K_{(0,t)}[u_*](t,x)+F^*(t,x,u_*(t,x),p,X)\le0\text{)}
\end{align*}
for all $(t,x)\in(a,T]\times O$ and $(p,X)\in\widetilde{\mathcal{P}}^+u^*(t,x)$ (resp. $\widetilde{\mathcal{P}}^-u_*(t,x)$).
\end{proposition}

We finish this section by presenting two types of stability of a solution.
The proofs are simple modifications of those of \cite[Theorems 5.1 and 5.2]{GigaNamba2017}, so we do not give the details here.
We note that the analogue of the vanishing viscosity can be obtained by a similar argument.

Let $A$ be a set in $\mathbf{R}^N$ with $N\ge1$.
For a sequence of functions $h_\sigma:A\to\mathbf{R}$, where $\sigma\in\mathbf{R}$, the upper half-relaxed limit $\Limsup_{\sigma\to\sigma'}h_\sigma$ and the lower half-relaxed limit $\Liminf_{\sigma\to\sigma'}h_\sigma$ are defined as
$$
(\Limsup_{\sigma\to\sigma'}h_\sigma)(x):=\lim_{\rho\to0}\sup\{h_\sigma(y)\mid y\in A\cap\overline{B(x,\rho)},\ 0<|\sigma-\sigma'|<\rho\}
$$
and $(\Liminf_{\sigma\to\sigma'}h_\sigma)(x):=-(\Limsup_{\sigma\to\sigma'}(-h_\sigma))(x)$ for $x\in A$, respectively.

\begin{proposition}[Stability]\label{p:Stabilityoflimit}
(i) Let $F_\sigma:W\to\mathbf{R}$ be a function such that $-\infty<(F_\sigma)_*\le (F_\sigma)^*<+\infty$ in $W$ for each $\sigma>0$.
Assume that 
$$
F_*\le \liminf_{\varepsilon\to0}{}_{*}(F_\sigma)_*\quad(\text{resp. }F^*\ge \limsup_{\varepsilon\to0}{}^{*}(F_\sigma)^*)\quad\text{in $W$.}
$$
For each $\sigma>0$ let $u_\sigma$ be a subsolution (resp. supersolution) of 
\begin{equation*}
\partial_t^\alpha u_\sigma+F_\sigma(t,x,u_\sigma, \nabla u_\sigma, \nabla^2u_\sigma)=0\quad\text{in $(a,T]\times O$.}
\end{equation*}
Then $\Limsup_{\sigma\to0}u_\sigma$ (resp. $\Liminf_{\sigma\to0}u_\sigma$) is a subsolution (resp. supersolution) of \eqref{e:master2} in $(a,T]\times O$ provided that $\Limsup_{\sigma\to0}u_\sigma<+\infty$ (resp. $\Liminf_{\sigma\to0}u_\sigma>-\infty$) in $[0,T]\times O$.

(ii) For each $\alpha\in(0,1)$ let $u_\alpha$ be a subsolution (resp. supersolution) of \eqref{e:master2} in $(a,T]\times O$ where the order of the Caputo time fractional derivative is $\alpha$.
Let $\beta\in(0,1]$.
Then $\Limsup_{\alpha\to\beta}u_\alpha$ (resp. $\Liminf_{\alpha\to\beta}u_\alpha$) is a subsolution (resp. supersolution) of \eqref{e:master2} in $(a,T]\times O$ where the order of the Caputo time fractional derivative is $\beta$, provided that $\Limsup_{\alpha\to\beta}u_\alpha<+\infty$ (resp. $\Liminf_{\alpha\to\beta}u_\alpha>-\infty$) in $[0,T]\times O$.
\end{proposition}

\section{Existence and uniqueness for the Cauchy-Dirichlet problem}
The goal of this section is to obtain a unique existence of a solution for the Cauchy-Dirichlet problem of the form
\begin{equation}\label{e:dirichlet}
  \begin{cases}
    \partial_t^\alpha u+F(t,x,u,\nabla u,\nabla^2u)=0\quad&\text{in $(0,T]\times\Omega$,}\\
    u=0\quad&\text{on $(0,T]\times\partial\Omega$,}\\
    u|_{t=0}=u_0\quad&\text{in $\overline{\Omega}$.}    
  \end{cases}
\end{equation}
Here $\Omega$ is a bounded domain in $\mathbf{R}^d$.
The boundary condition is interpreted in the strong sense.  
\begin{definition}
A function $u:[0,T]\times\overline{\Omega}\to\mathbf{R}$ is a (viscosity) subsolution (resp. supersolution) of  \eqref{e:dirichlet} if it is a (viscosity) subsolution (resp. supersolution) of \eqref{e:master2} in $(0,T]\times\Omega$, $u^*\le 0$ (resp. $u_*>-\infty$) on $(0,T]\times\partial\Omega$, and $u^*(0,\cdot)\le u_0$ (resp. $u_*(0,\cdot)\ge u_0$) in $\overline{\Omega}$.
If a function $u:[0,T]\times\overline{\Omega}\to\mathbf{R}$ is both a (viscosity) sub- and supersolution of \eqref{e:dirichlet}, then $u$ is called a (viscosity) solution of \eqref{e:dirichlet}.
\end{definition}
We list our assumptions on $F$ and $u_0$.
\begin{itemize}
\item[(A1)] $F\in C((0,T]\times\Omega\times\mathbf{R}\times\mathbf{R}^d\times\mathbf{S}^d)$,
\item[(A2)] for all $(t,x,p,X)\in(0,T]\times\Omega\times\mathbf{R}^d\times\mathbf{S}^d$
$$
F(t,x,w_1,p,X)-F(t,x,w_2,p,X)\ge0\quad\text{if $w_1\ge w_2$},
$$
\item[(A3)] there is a modulus of continuity $\omega:[0,\infty]\to[0,\infty]$  that satisfies $\lim_{r\searrow0}\omega(r)=0$ such that
\begin{align*}
&F(t,x,w,\sigma^{-1}(x-y),-Y)-F(t,y,w,\sigma^{-1}(x-y),X)\\
&\le\omega(|x-y|(1+\sigma^{-1}|x-y|))
\end{align*}
for all $(t,x,y,w)\in(0,T]\times\Omega\times\Omega\times\mathbf{R}$, $X,Y\in\mathbf{S}^d$, and $\sigma>0$ satisfying
\begin{equation*}
\left(\begin{array}{cc}X&O\\ O&Y\end{array}\right)\le\frac{2}{\sigma}\left(\begin{array}{cc}I & -I\\ -I&I\end{array}\right),
\end{equation*}
\item[(A4)] $u_0\in C(\overline{\Omega})$ that satisfies $u_0=0$ on $\partial\Omega$.
\end{itemize}

As is known, (A1) and (A3) yield the degenerate ellipticity \eqref{e:degenerate} of $F$.
We note that the assumptions (A1)-(A3) are easily weakened.
For discussion of assumptions, see \cite{CrandallIshiiLions1992}, \cite{Giga2006} and references therein.

\subsection{Comparison principle}
In this subsection we prove a comparison principle. For this purpose we present a key lemma inspired by \cite[Proposition 5.1]{Ishii1989}, \cite[Proposition IV.1]{IshiiLions1990}, and \cite{JakobsenKarlsen2006}. The proof of the lemma is postponed to the next section. For bounded functions $u,v:[0,T]\times\Omega\to\mathbf{R}$ and a parameter $\varepsilon>0$, let $u^\varepsilon$ and $v_\varepsilon$ denote the sup- and inf-convolution in space of $u$ and $v$, respectively. Namely,
\begin{align*}
&u^\varepsilon(t,x)=\sup_{x'\in\Omega}\{u(t,x')-\varepsilon^{-1}|x-x'|^2\}\quad\text{and}\\
&v_\varepsilon(t,x)=\inf_{x'\in\Omega}\{v(t,x')+\varepsilon^{-1}|x-x'|^2\}
\end{align*}
for $(t,x)\in [0,T]\times\Omega$.
We write $a\vee b=\max\{a,b\}$ and $a\wedge b=\min\{a,b\}$ for $a,b\in\mathbf{R}$.

\begin{lemma}\label{l:ishii}
Assume (A1) and (A2).
Let $u,v:[0,T]\times\Omega\to\mathbf{R}$ be a bounded usc subsolution and a bounded lsc supersolution of \eqref{e:master2} in $(0,T]\times\Omega$, respectively.
Let $\varphi\in C^2([0,T]\times\Omega\times\Omega)$ and set $\Omega_\varepsilon:=\{x\in\Omega\mid \dist(x,\partial\Omega)>M\varepsilon^{1/2}\}$ with $M:=(2\sup_{[0,T]\times\Omega}|u|)^{1/2}\vee(2\sup_{[0,T]\times\Omega}|v|)^{1/2}$, where $\dist(x,\partial \Omega):=\inf_{y\in\partial \Omega}|x-y|$ and $\varepsilon>0$.
Let $(\bar{t},\bar{x},\bar{y})\in(0,T]\times\Omega_\varepsilon\times\Omega_\varepsilon$ be such that
$$
\max_{(t,x,y)\in[0,T]\times\overline{\Omega_\varepsilon}\times\overline{\Omega_\varepsilon}}(u^\varepsilon(t,x)-v_\varepsilon(t,y)-\varphi(t,x,y))=u^\varepsilon(\bar{t},\bar{x})-v_\varepsilon(\bar{t},\bar{y})-\varphi(\bar{t},\bar{x},\bar{y}).
$$
Then there exist two matrices $X,Y\in\mathbf{S}^d$ satisfying 
\begin{equation}\label{e:comparison1}
-\frac{2}{\varepsilon}\left(\begin{array}{cc}I&O\\ O&I\end{array}\right)\le\left(\begin{array}{cc}X&O\\ O&Y\end{array}\right)\le\nabla_{x,y}^2\varphi(\bar{t},\bar{x},\bar{y}).
\end{equation}
such that $K_{(0,\bar{t})}[u^\varepsilon](\bar{t},\bar{x})-K_{(0,\bar{t})}[v_\varepsilon](\bar{t},\bar{y})$ exist and 
\begin{equation}\label{e:comparison2}
\begin{split}
&J[u^\varepsilon](\bar{t},\bar{x})-J[v_\varepsilon](\bar{t},\bar{y})+K_{(0,\bar{t})}[u^\varepsilon](\bar{t},\bar{x})-K_{(0,\bar{t})}[v_\varepsilon](\bar{t},\bar{y})\\
&+F_\varepsilon(\bar{t},\bar{x},u^\varepsilon(\bar{t},\bar{x}),\nabla_x\varphi(\bar{t},\bar{x},\bar{y}),X)-F^\varepsilon(\bar{t},\bar{y},v_\varepsilon(\bar{t},\bar{y}),-\nabla_y\varphi(\bar{t},\bar{x},\bar{y}),-Y)\le0.
\end{split}
\end{equation}
Here $I$ denotes the $d\times d$ identity matrix,
\begin{align*}
&F_\varepsilon(t,x,w,p,X)=\min\{F(t,x',w,p,X)\mid x'\in\overline{B(x,M\varepsilon^{1/2})}\},\quad\text{and}\\
&F^\varepsilon(t,x,w,p,X)=\max\{F(t,x',w,p,X)\mid x'\in\overline{B(x,M\varepsilon^{1/2})}\}.
\end{align*}
\end{lemma}

For the reader's convenience we give basic properties of the sup- and inf-convolution here. See, e.g., \cite{BardiCapuzzoDolcetta1997}, \cite{CannarsaSinestrari2004}, and \cite{LasryLions1986} for the proof.

\begin{proposition}
Let $D$ be a bounded domain in $\mathbf{R}^N$ with $N\ge1$ and $f:D\to\mathbf{R}$ be a bounded function.
For $\varepsilon>0$ let $f^\varepsilon,f_\varepsilon$ be real-valued functions on $\mathbf{R}^N$ defined by
\begin{align*}
&f^\varepsilon(z)=\sup_{z'\in D}\{f(z')-\varepsilon^{-1}|z-z'|^2\}\quad\text{and}\\
&f_\varepsilon(z)=\inf_{z'\in D}\{f(z')+\varepsilon^{-1}|z-z'|^2\}.
\end{align*}
Then the following properties hold:
\begin{itemize}
\item[(i)] $f_\varepsilon=-(-f)^\varepsilon$ on $\mathbf{R}^N$,
\item[(ii)] $f\le f^\varepsilon\le \sup_D|f|$ on $D$,
\item[(iii)]  Let $C>0$ be such that $C\ge(2\sup_D|f|)^{1/2}$. Set $D_\varepsilon=\{y\in D\mid \dist(y,\partial D)>C \varepsilon^{1/2}\}$. For $\varepsilon>0$ such that $D_\varepsilon\neq\emptyset$ and $x\in D_\varepsilon$ there exists $x'\in\overline{B(x,C\varepsilon^{1/2})}$ such that $f^\varepsilon(x)=f(x')-\varepsilon^{-1}|x-x'|^2$,
\item[(iv)] $f^\varepsilon$ is semiconvex, that is, $\nabla^2 f^\varepsilon\ge-(2\varepsilon)^{-1}I_N$ in $D$ (in the sense of distributions), where $I_N$ denotes the $N\times N$ identity matrix.
\end{itemize}
\end{proposition}

\begin{theorem}[Comparison principle]\label{t:comparison}
Assume (A1), (A2), and (A3). Let $u,v:[0,T]\times\overline{\Omega}\to\mathbf{R}$ be a bounded usc subsolution and a bounded lsc supersolution of \eqref{e:master2} in $(0,T]\times\Omega$, respectively.
If $u\le v$ on $(\{0\}\times\overline{\Omega})\cup([0,T]\times\partial\Omega)$, then $u\le v$ on $[0,T]\times\overline{\Omega}$.
\end{theorem}

\begin{proof}
We suppose that $\max_{[0,T]\times\overline{\Omega}}(u-v)=:\theta>0$ and shall get a contradiction.

For $\sigma>0$ we define a function $\Phi$ on $[0,T]\times\overline{\Omega}\times\overline{\Omega}$ by
$$
\Phi(t,x,y)=u(t,x)-v(t,y)-\frac{|x-y|^2}{\sigma}.
$$
Since $\Phi$ is usc on the compact set $[0,T]\times\overline{\Omega}\times\overline{\Omega}$, there is a maximum point $(t_\sigma,x_\sigma,y_\sigma)\in[0,T]\times\overline{\Omega}\times\overline{\Omega}$ of $\Phi$.
It is standard \cite[Lemma 3.1]{CrandallIshiiLions1992} that $(t_\sigma,x_\sigma,y_\sigma)$ converges to a $(\hat{t},\hat{x},\hat{x})\in(0,T]\times\Omega\times\Omega$ such that $(u-v)(\hat{t},\hat{x})=\theta$ as $\sigma\to0$ by taking a subsequence if necessary and that $\lim_{\sigma\to0}\sigma^{-1}|x_\sigma-y_\sigma|^2=0$.
Notice that $\hat{t}\neq0$ and $\hat{x}\not\in\partial\Omega$ since, otherwise, it would contradict the assumption that $u\le v$ on $(\{0\}\times\overline{\Omega})\cup([0,T]\times\partial\Omega)$.
We may assume that all maximum points of $\Phi$ and $(\hat{t},\hat{x},\hat{x})$ are in $(0,T]\times\Omega_\varepsilon\times\Omega_\varepsilon$ by letting $\sigma$ and $\varepsilon$ smaller.

Let $(t_{\sigma,\varepsilon},x_{\sigma,\varepsilon},y_{\sigma,\varepsilon})$ be a maximum point of
$$
\Phi_\varepsilon(t,x,y):=u^\varepsilon(t,x)-v_\varepsilon(t,y)-\frac{|x-y|^2}{\sigma}
$$
on $[0,T]\times\overline{\Omega_\varepsilon}\times\overline{\Omega_\varepsilon}$.
As in \cite[Section III]{IshiiLions1990} (or Proposition \ref{p:basic} of this paper), it follows that $(t_{\sigma,\varepsilon},x_{\sigma,\varepsilon},y_{\sigma,\varepsilon})$ converges to a maximum point of $\Phi$ as $\varepsilon\to0$.
Hence $(t_{\sigma,\varepsilon},x_{\sigma,\varepsilon},y_{\sigma,\varepsilon})\in(0,T]\times\Omega_\varepsilon\times\Omega_\varepsilon$ for sufficiently small $\varepsilon$.
We denote the limit of $(t_{\sigma,\varepsilon},x_{\sigma,\varepsilon},y_{\sigma,\varepsilon})$ as $\sigma,\varepsilon\to0$ by $(\hat{t},\hat{x},\hat{x})$, although it is not necessarily the same as the above one.
Since
\begin{align*}
\Phi_\varepsilon(t_{\sigma,\varepsilon},x_{\sigma,\varepsilon},y_{\sigma,\varepsilon})
&=\max_{(t,x,y)\in[0,T]\times\overline{\Omega_\varepsilon}\times\overline{\Omega_\varepsilon}}\Phi_\varepsilon(t,x,y)\\
&\ge\max_{(t,x,y)\in[0,T]\times\overline{\Omega_\varepsilon}\times\overline{\Omega_\varepsilon}}\Phi(t,x,y)\\
&\ge\max_{(t,x)\in[0,T]\times\overline{\Omega_\varepsilon}}\Phi(t,x,x)=\theta>0,
\end{align*}
we have $u^\varepsilon(t_{\sigma,\varepsilon},x_{\sigma,\varepsilon})-v_\varepsilon(t_{\sigma,\varepsilon},y_{\sigma,\varepsilon})\ge\theta>0$ and so $u^\varepsilon(t_{\sigma,\varepsilon},x_{\sigma,\varepsilon})\ge v_\varepsilon(t_{\sigma,\varepsilon},y_{\sigma,\varepsilon})$.

We now apply Lemma \ref{l:ishii} with $(\bar{t},\bar{x},\bar{y})=(t_{\sigma,\varepsilon},x_{\sigma,\varepsilon},y_{\sigma,\varepsilon})$ and $\varphi(t,x,y)=\sigma^{-1}|x-y|^2$ to obtain two matrices $X,Y\in\mathbf{S}^d$ satisfying
\begin{equation*}
  -\frac{2}{\varepsilon}\left(\begin{array}{cc}I&O\\ O&I\end{array}\right)
  \le\left(
  \begin{array}{cc}
  X & O\\
  O & Y
  \end{array}\right) \le\frac{2}{\sigma}\left(
  \begin{array}{cc}
  I & -I\\
  -I & I
  \end{array}\right)
\end{equation*}
such that
\begin{equation}\label{e:comparison_1}
\begin{split}
&J[u^\varepsilon](t_{\sigma,\varepsilon},x_{\sigma,\varepsilon})-J[v_\varepsilon](t_{\sigma,\varepsilon},y_{\sigma,\varepsilon})+K_{(0,t_{\sigma,\varepsilon})}[u^\varepsilon](t_{\sigma,\varepsilon},x_{\sigma,\varepsilon})-K_{(0,t_{\sigma,\varepsilon})}[v_\varepsilon](t_{\sigma,\varepsilon},y_{\sigma,\varepsilon})\\
&+F_\varepsilon(t_{\sigma,\varepsilon},x_{\sigma,\varepsilon},u^\varepsilon(t_{\sigma,\varepsilon},x_{\sigma,\varepsilon}),p,X)-F^\varepsilon(t_{\sigma,\varepsilon},y_{\sigma,\varepsilon},v_\varepsilon(t_{\sigma,\varepsilon},y_{\sigma,\varepsilon}),p,-Y)\le0,
\end{split}
\end{equation}
where $p:=2\sigma^{-1}(x_{\sigma,\varepsilon}-y_{\sigma,\varepsilon})$.
Let $(x'_{\sigma,\varepsilon},y'_{\sigma,\varepsilon})\in\overline{B(x_{\sigma,\varepsilon},M\varepsilon^{1/2})}\times\overline{B(y_{\sigma,\varepsilon},M\varepsilon^{1/2})}$ such that
\begin{align*}
&u^\varepsilon(0,x_{\sigma,\varepsilon})=u(0,x'_{\sigma,\varepsilon})-\varepsilon^{-1}|x_{\sigma,\varepsilon}-x'_{\sigma,\varepsilon}|^2\quad\text{and}\\
&v_\varepsilon(0,y_{\sigma,\varepsilon})=v(0,y'_{\sigma,\varepsilon})+\varepsilon^{-1}|y_{\sigma,\varepsilon}-y'_{\sigma,\varepsilon}|^2.
\end{align*}
Since $x'_{\sigma,\varepsilon}$ and $y'_{\sigma,\varepsilon}$ converge to $\hat{x}$ as $\sigma,\varepsilon\to0$, we find that
\begin{align*}
\limsup_{\sigma,\varepsilon\to0}(u^\varepsilon(0,x_{\sigma,\varepsilon})-v_\varepsilon(0,y_{\sigma,\varepsilon}))
\le\limsup_{\sigma,\varepsilon\to0}(u(0,x'_{\sigma,\varepsilon})-v(0,y'_{\sigma,\varepsilon}))
\le(u-v)(0,\hat{x})
\end{align*}
and so that
\begin{align*}
\liminf_{\sigma,\varepsilon\to0}(J[u^\varepsilon](t_{\sigma,\varepsilon},x_{\sigma,\varepsilon})-J[v_\varepsilon](t_{\sigma,\varepsilon},y_{\sigma,\varepsilon}))
&\ge\liminf_{\sigma,\varepsilon\to0}\frac{\theta-u^\varepsilon(0,x_{\sigma,\varepsilon})+v_\varepsilon(0,y_{\sigma,\varepsilon})}{t_{\sigma,\varepsilon}^\alpha\Gamma(1-\alpha)}\\
&\ge\frac{\theta-(u-v)(0,\hat{x})}{\hat{t}^\alpha\Gamma(1-\alpha)}.
\end{align*}
Since $\Phi_\varepsilon(t_{\sigma,\varepsilon},x_{\sigma,\varepsilon},y_{\sigma,\varepsilon})\ge\Phi_\varepsilon(t_{\sigma,\varepsilon}-\tau,x_{\sigma,\varepsilon},y_{\sigma,\varepsilon})$, that is, $ u^\varepsilon(t_{\sigma,\varepsilon},x_{\sigma,\varepsilon})-u^\varepsilon(t_{\sigma,\varepsilon}-\tau,x_{\sigma,\varepsilon})-v_\varepsilon(t_{\sigma,\varepsilon},y_{\sigma,\varepsilon})+v_\varepsilon(t_{\sigma,\varepsilon}-\tau,y_{\sigma,\varepsilon})\ge0$ for all $\tau\in[0,t_{\sigma,\varepsilon}]$, we have
\begin{equation}\label{e:comparison_1.5}
K_{(0,t_{\sigma,\varepsilon})}[u^\varepsilon](t_{\sigma,\varepsilon},x_{\sigma,\varepsilon})-K_{(0,t_{\sigma,\varepsilon})}[v_\varepsilon](t_{\sigma,\varepsilon},y_{\sigma,\varepsilon})\ge0.
\end{equation}
Let $(x_{\sigma,\varepsilon}'',y_{\sigma,\varepsilon}'')\in\overline{B(x_{\sigma,\varepsilon},M\varepsilon^{1/2})}\times\overline{B(y_{\sigma,\varepsilon},M\varepsilon^{1/2})}$ such that
\begin{align*}
&F_\varepsilon(t_{\sigma,\varepsilon},x_{\sigma,\varepsilon},u^\varepsilon(t_{\sigma,\varepsilon},x_{\sigma,\varepsilon}),p,X)=F(t_{\sigma,\varepsilon},x_{\sigma,\varepsilon}'',u^\varepsilon(t_{\sigma,\varepsilon},x_{\sigma,\varepsilon}),p,X)\quad\text{and}\\
&F^\varepsilon(t_{\sigma,\varepsilon},y_{\sigma,\varepsilon},v_\varepsilon(t_{\sigma,\varepsilon},y_{\sigma,\varepsilon}),p,-Y)=F(t_{\sigma,\varepsilon},y_{\sigma,\varepsilon}'',v_\varepsilon(t_{\sigma,\varepsilon},y_{\sigma,\varepsilon}),p,-Y).
\end{align*}
By (A2) and (A3) we see that 
\begin{equation}\label{e:comparison_2}
\begin{split}
&F_\varepsilon(t_{\sigma,\varepsilon},x_{\sigma,\varepsilon},u^\varepsilon(t_{\sigma,\varepsilon},x_{\sigma,\varepsilon}),p,X)-F^\varepsilon(t_{\sigma,\varepsilon},y_{\sigma,\varepsilon},v_\varepsilon(t_{\sigma,\varepsilon},y_{\sigma,\varepsilon}),p,-Y)\\
&\ge -\omega(|x''_{\sigma,\varepsilon}-y''_{\sigma,\varepsilon}|(1+\sigma^{-1}|x''_{\sigma,\varepsilon}-y''_{\sigma,\varepsilon}|).
\end{split}
\end{equation}
Therefore, taking the limit infimum in \eqref{e:comparison_1} as $\sigma,\varepsilon\to0$ after applying \eqref{e:comparison_1.5} and \eqref{e:comparison_2} yields
$$
\frac{\theta-(u-v)(0,\hat{x})}{\hat{t}^\alpha\Gamma(1-\alpha)}\le0
$$
since $\lim_{\sigma\to0}\lim_{\varepsilon\to0}\omega(|x''_{\sigma,\varepsilon}-y''_{\sigma,\varepsilon}|(1+\sigma^{-1}|x''_{\sigma,\varepsilon}-y''_{\sigma,\varepsilon}|)=0$.
However, this is a contradiction since $\theta>0$ and $(u-v)(0,\hat{x})\le0$.
\end{proof}

\subsection{Existence of a solution}
We show an existence theorem by Perron's method.

\begin{theorem}[Existence of a solution]\label{t:existence}
Assume that (A1), (A2), (A3), and (A4). 
Let $u_-,u_+:[0,T]\times\overline{\Omega}\to\mathbf{R}$ be a subsolution and a supersolution of \eqref{e:dirichlet} with $(u_-)_*>-\infty$, $(u_+)^*<+\infty$ in $[0,T]\times\overline{\Omega}$.
Assume that $(u_-)_*=(u_+)^*=0$ on $[0,T]\times\partial\Omega$ and $(u_-)_*(0,\cdot)=(u_+)^*(0,\cdot)=u_0$ on $\overline{\Omega}$.
Then there exists a solution $u\in C([0,T]\times\overline{\Omega})$ of \eqref{e:dirichlet} that satisfies $u_-\le u\le u_+$ in $[0,T]\times\overline{\Omega}$.
\end{theorem}

\begin{lemma}\label{l:closedness}
Assume (A1).
Let $S$ be a nonempty set of subsolutions (resp. supersolutions) of \eqref{e:master2} in $(0,T]\times\Omega$.
Set 
$$
u(t,x):=\sup\{v(t,x)\mid v\in S\}\quad(\text{resp. }\inf\{v(t,x)\mid v\in S\})
$$
for $(t,x)\in [0,T]\times\Omega$.
Then $u$ is a subsolution (resp. supersolution) of \eqref{e:master2} in $(0,T]\times\Omega$ provided that $u^*<+\infty$ (resp. $u_*>-\infty$) in $[0,T]\times\Omega$.
\end{lemma}

\begin{lemma}\label{l:bump}
Assume (A1) and (A2).
Let $u_+:[0,T]\times\Omega\to\mathbf{R}$ be a supersolution of \eqref{e:master2} in $(0,T]\times\Omega$.
Let $S$ be a nonempty set of subsolutions $u$ of \eqref{e:master2} in $(0,T]\times\Omega$ such that if $u\in S$, then $u\le u_+$ in $[0,T]\times\Omega$.
If $u\in S$ is not a supersolution of \eqref{e:master2} in $(0,T]\times\Omega$ with $u_*>-\infty$ in $[0,T)\times\Omega$, then there exists a function $w\in S$ and a point $(s,y)\in(0,T]\times\Omega$ such that $u(s,y)<w(s,y)$.
\end{lemma}

Theorem \ref{t:existence} follows from lemmas \ref{l:closedness} and \ref{l:bump} as in \cite{Ishii1987}.
Both lemmas can be proved by slightly modifying proofs of \cite[Lemma 4.1 and Theorem 4.2]{GigaNamba2017}.
We only give a simpler proof for Lemma \ref{l:bump}.

\begin{proof}[Proof of Lemma \ref{l:bump}]
Since $u\in S$ is not a supersolution of \eqref{e:master2} in $(0,T]\times\Omega$, there exists $((\hat{t},\hat{x}),\varphi)\in((0,T]\times\Omega)\times (C^{1,2}((0,T]\times\Omega)\cap C([0,T]\times\Omega))$ such that $u_*-\varphi$ attains a strict zero minimum at $(\hat{t},\hat{x})$ and
\begin{equation}\label{e:existence1}
J[\varphi](\hat{t},\hat{x})+K_{(0,\hat{t})}[\varphi](\hat{t},\hat{x})+F(\hat{t},\hat{x},\varphi(\hat{t},\hat{x}),\nabla\varphi(\hat{t},\hat{x}),\nabla^2\varphi(\hat{t},\hat{x}))<0.
\end{equation}
Proposition \ref{p:propertyK} and the dominated convergence theorem imply that $(t,x)\mapsto K_{(0,t)}[\varphi](t,x)$ is continuous in $(0,T]\times\Omega$.
Thus, for sufficiently small $\rho>0$ and $r>0$, we have
\begin{equation}\label{e:existence2}
J[\varphi](t,x)+K_{(0,t)}[\varphi](t,x)+F(t,x,\varphi(t,x)+\rho,\nabla\varphi(t,x),\nabla^2\varphi(t,x))\le0
\end{equation}
for $(t,x)\in B_{2r}:=B((\hat{t},\hat{x}),2r)\cap((0,T]\times\Omega)$.

It is easy to see that $\varphi\le u_*\le (u_+)_*$ in $[0,T]\times\Omega$.
We notice that $\varphi<(u_+)_*$ at $(\hat{t},\hat{x})$.
Indeed, if $\varphi=(u_+)_*$ at $(\hat{t},\hat{x})$, then $\min_{[0,T]\times\Omega}((u_+)_*-\varphi)=((u_+)_*-\varphi)(\hat{t},\hat{x})$.
Since $u_+$ is a supersolution of \eqref{e:master2} in $(0,T]\times\Omega$, the inequality \eqref{e:existence1} is contradictory.
Set $\lambda:=\frac{1}{2}((u_+)_*-\varphi)(\hat{x},\hat{t})>0$.
Since $(u_+)_*-\varphi$ is lsc, we see that $\varphi+\lambda\le (u_+)_*$ in $B_{2r}$ by letting $r$ smaller if necessary. 
Note thus that $\varphi+\lambda\le u_+$ in $B_{2r}$. 
Note also that there is $\lambda'\in(0,\lambda\wedge\rho)$ such that $\varphi+2\lambda'\le u_*\le u^*$ in $([0,T]\times\overline{B(\hat{x},2r)})\setminus B_r$ since $u_*>\varphi$ in $([0,T]\times\Omega)\setminus\{(\hat{t},\hat{x})\}$.

Let a function $w$ on $[0,T]\times\Omega$ be defined by
\begin{equation*}
  w=\begin{cases}
    u\vee (\varphi+\lambda')\quad&\text{in $B_r$,}\\
    u\quad&\text{in $([0,T]\times\Omega)\setminus B_r$.}
  \end{cases}
\end{equation*} 
We claim that $w\in S$.
Clearly, $w\le u_+$ in $[0,T]\times\Omega$, so it suffices to prove that $w$ is a subsolution of \eqref{e:master2} in $(0,T]\times\Omega$.
To this end, we take $((\hat{s},\hat{y}),\psi)\in((0,T]\times\Omega)\times(C^{1,2}((0,T]\times\Omega)\cap C([0,T]\times\Omega))$ that satisfies $\max_{[0,T]\times\Omega}(w^*-\psi)=(w^*-\psi)(\hat{s},\hat{y})=0$.

Assume that $w^*=u^*$ at $(\hat{s},\hat{y})$.
Then $\max_{[0,T]\times\Omega}(u^*-\psi)=(u^*-\psi)(\hat{t},\hat{x})$.
Since $u$ is a subsolution of \eqref{e:master2} in $(0,T]\times\Omega$, we get the desired inequality.

Assume that $w^*=\varphi+\lambda'$ at $(\hat{s},\hat{y})$.
Notice then that $(\hat{s},\hat{y})\in B_r$.
The definition of $w$ implies that $w(\hat{s}-\tau,\hat{y})\ge\varphi(\hat{s}-\tau,\hat{y})+\lambda'$ for all $\tau\in[0,\hat{s}]$.
Hence we see that
\begin{align*}
\psi(\hat{s},\hat{y})-\psi(\hat{s}-\tau,\hat{y})-\varphi(\hat{s},\hat{y})+\varphi(\hat{s}-\tau,\hat{y})
&=\lambda'-\psi(\hat{s}-\tau,\hat{y})+\varphi(\hat{s}-\tau,\hat{y})\\
&\le-\psi(\hat{s}-\tau,\hat{y})+w^*(\hat{s}-\tau,\hat{y})\\
&\le0
\end{align*}
for all $\tau\in[0,\hat{s}]$.
This yields $J[\psi](\hat{s},\hat{y})\le J[\varphi](\hat{s},\hat{y})$ and $K_{(0,\hat{s})}[\psi](\hat{s},\hat{y})\le K_{(0,\hat{s})}[\varphi](\hat{s},\hat{y})$.
We also see that $\varphi+\lambda'-\psi$ attains a local zero maximum at $(\hat{s},\hat{y})$, so that $(\nabla\psi,\nabla^2\psi)=(\nabla\varphi,\nabla^2\varphi)$ at $(\hat{s},\hat{y})$.
In consequence, by \eqref{e:existence2} and (A2) it follows that 
\begin{align*}
&J[\psi](\hat{s},\hat{y})+K_{(0,\hat{s})}[\psi](\hat{s},\hat{y})+F(\hat{s},\hat{y},\psi(\hat{s},\hat{y}),\nabla\psi(\hat{s},\hat{y}),\nabla^2\psi(\hat{s},\hat{y}))\\
&\le J[\varphi](\hat{s},\hat{y})+K_{(0,\hat{s})}[\varphi](\hat{s},\hat{y})+F(\hat{s},\hat{y},\varphi(\hat{s},\hat{y})+\lambda',\nabla\varphi(\hat{s},\hat{y}),\nabla^2\varphi(\hat{s},\hat{y}))\le0,
\end{align*}
which asserts our claim.

Let $(t_\sigma,x_\sigma)$ be a sequence such that $(t_\sigma,x_\sigma,u(t_\sigma,x_\sigma))\to (\hat{t},\hat{x},u_*(\hat{t},\hat{x}))$ as $\sigma\searrow0$. Then we have
$$
\liminf_{\sigma\to0}(w(t_\sigma,x_\sigma)-u(t_\sigma,x_\sigma))\ge \lim_{\sigma\to0}(\varphi(t_\sigma,x_\sigma)+\lambda'-u(t_\sigma,x_\sigma))=\lambda'>0.
$$
This means that there is a point $(t,x)\in[0,T]\times\Omega$ such that $w(t,x)>u(t,x)$.
\end{proof}

Finally, we would like to point out that sub- and supersolutions assumed in Theorem \ref{t:existence} are obtained by slightly extending the construction method in case of $\alpha=1$ given, e.g., by Demengel \cite{Demengel}. As an example let us consider the following simple equation under the same initial-boundary condition as \eqref{e:dirichlet}.
\begin{equation*}
\partial_t^\alpha u-\Delta u=0\quad\text{in $(0,T]\times\Omega$}.
\end{equation*}
For any fixed parameters $(s,y)\in(\{0\}\times\overline{\Omega})\cup((0,T]\times\partial\Omega)$ and $\varepsilon>0$ we define
\begin{equation*}
u_-^{s,y,\varepsilon}(t,x)=
	\begin{cases}
		g(s,y)-\varepsilon-C_1(\rho_1(t)+\rho_2(x))&\quad\text{if $(s,y)\in(0,T]\times\partial\Omega$},\\
		g(0,y)-\varepsilon-C_2(\frac{1}{\Gamma(1+\alpha)}t^\alpha+\frac{1}{2d}|x-y|^2)&\quad\text{if $(s,y)\in\{0\}\times\overline{\Omega}$}
	\end{cases}
\end{equation*}
for $(x,t)\in[0,T]\times\overline{\Omega}$. Here, $C_1$ and $C_2$ are sufficiently large constants, $g$ is a function on $(\{0\}\times\overline{\Omega})\cup((0,T]\times\partial\Omega)$ that represents the initial-boundary condition,
$$
\rho_1(t)=\frac{\alpha t^{1+\alpha}-(1+\alpha)st^\alpha+s^{1+\alpha}}{\Gamma(2+\alpha)T},
$$ 
and $\rho_2$ is a real-valued function on $\overline{\Omega}$ satisfying 
$$
-\Delta \rho_2\ge1\quad\text{in $\Omega$ (in the viscosity sense)},\quad \rho_2>0\quad\text{on $\overline{\Omega}\setminus\{y\}$},\quad \rho_2(y)=0.
$$
We note that, in the definition of $u_-^{s,y,\varepsilon}$, only terms $\rho_1$ and $t^\alpha$ are extended from one handled in \cite{CrandallKocanLionsSwiech}. By a very similar argument (without additional step) it can be proved that $u_-^{s,y,\varepsilon}$ is a bounded subsolution of the initial-boundary problem that satisfies $u_-^{s,y,\varepsilon}(s,y)=g(s,y)-\varepsilon$. Therefore Lemma \ref{l:closedness} implies that $u_-:=\sup\{u_-^{s,y,\varepsilon}\mid(s,y)\in(\{0\}\times\overline{\Omega})\cup((0,T]\times\partial\Omega),\varepsilon>0\}$ is a desired subsolution. The same change also enables us to construct sub- and supersolutions for equations with more general $F$ handled in \cite{Demengel}.

\section{Proof of Lemma \ref{l:ishii}}
This section is devoted to prove Lemma \ref{l:ishii} which played the important role in proving the comparison principle (Theorem \ref{t:comparison}).
We follow conventional procedures. That is, the proof consists of finding appropriate semijets using the regularization of a solution by sup- and inf-convolutions in space-time (See \cite{ChenGigaGoto1991}, \cite{ChenGigaGoto19912}, and \cite{IshiiLions1990} for example). 
We first prove a weaker case $\bar{t}\neq T$.

\begin{proposition}\label{p:weakishii}
Let $u,v$ and $F$ be as in Lemma \ref{l:ishii}.
Assume that $(\bar{t},\bar{x},\bar{y})\in(0,T)\times\Omega_\varepsilon\times\Omega_\varepsilon$.
Then the same conclusion as Lemma \ref{l:ishii} holds.  
\end{proposition}

For functions $u,v:[0,T]\times\Omega\to\mathbf{R}$, parameters $\varepsilon>0$, and $\delta>0$, let $u^{\varepsilon,\delta}$ and $v_{\varepsilon,\delta}$ denote the sup- and inf-convolution in time of $u^\varepsilon$ and $v_\varepsilon$, respectively. Namely,
\begin{align*}
&u^{\varepsilon,\delta}(t,x)=\sup_{t'\in[0,T]}\{u^\varepsilon(t',x)-\delta^{-1}|t-t'|^2\}\quad\text{and}\\
&v_{\varepsilon,\delta}(t,x)=\inf_{t'\in[0,T]}\{v_\varepsilon(t',x)+\delta^{-1}|t-t'|^2\}
\end{align*}
for $[0,T]\times\Omega$.

\begin{proposition}\label{p:conv}
Assume (A1) and (A2).
Let $u,v:[0,T]\times\Omega\to\mathbf{R}$ be a bounded usc subsolution and a lsc supersolution of \eqref{e:master2} in $(0,T]\times\Omega$.
Then $u^{\varepsilon,\delta}$ and $v_{\varepsilon,\delta}$ are a subsolution and a supersolution of 
\begin{align*}
&\partial_t^\alpha u^{\varepsilon,\delta}+F_{\varepsilon,\delta}(t,x,u^{\varepsilon,\delta},\nabla u^{\varepsilon,\delta},\nabla^2 u^{\varepsilon,\delta})=\eta_\delta^u\quad\text{in $(M\delta^{1/2},T]\times\Omega_\varepsilon$}\quad\text{and}\\
&\partial_t^\alpha v_{\varepsilon,\delta}+F^{\varepsilon,\delta}(t,x,v_{\varepsilon,\delta},\nabla v_{\varepsilon,\delta},\nabla^2 v_{\varepsilon,\delta})=\eta_\delta^v\quad\text{in $(M\delta^{1/2},T]\times\Omega_\varepsilon$},
\end{align*}
respectively.
Here
\begin{align*}
&F_{\varepsilon,\delta}(t,x,w,p,X)=\min\{F(t',x',w,p,X) \mid |t-t'|\le M\delta^{1/2},x'\in\overline{B(x,M\varepsilon^{1/2})}\},\\
&F^{\varepsilon,\delta}(t,x,w,p,X)=\max\{F(t',x',w,p,X) \mid |t-t'|\le M\delta^{1/2},x'\in\overline{B(x,M\varepsilon^{1/2})}\},
\end{align*}
and $\eta_\delta^u,\eta_\delta^v$ are constants such that $\eta_\delta^u,\eta_\delta^v\to0$ as $\delta\to0$.
\end{proposition}

\begin{proof}
We only prove for subsolution since a similar argument applies to the other case.
The assertion that $K_{(0,\hat{t})}[u^{\varepsilon,\delta}](\hat{t},\hat{x})$ exists is a consequence that $u^{\varepsilon,\delta}$ is Lipschitz continuous with respect to time in $[0,T]$.

We take $((\hat{t},\hat{x}),\varphi)\in ((M\delta^{1/2},T]\times\Omega_\varepsilon)\times C^{1,2}((M\delta^{1/2},T]\times\Omega_\varepsilon)$ that satisfies 
$$
\max_{(M\delta^{1/2},T]\times\Omega_\varepsilon}(u^{\varepsilon,\delta}-\varphi)=(u^{\varepsilon,\delta}-\varphi)(\hat{t},\hat{x}).
$$
Let $(t',x')\in[\hat{t}-M\delta^{1/2},\hat{t}+M\delta^{1/2}]\times\overline{B(\hat{x},M\varepsilon^{1/2})}$ be such that
\begin{align}
&u^{\varepsilon,\delta}(\hat{t},\hat{x})=u^\varepsilon(t',\hat{x})-\delta^{-1}|\hat{t}-t'|^2\label{e:conv}\quad\text{and}\\
&u^\varepsilon(t',\hat{x})=u(t',x')-\varepsilon^{-1}|\hat{x}-x'|^2\label{e:conv1}.
\end{align}
Observe that
\begin{align*}
&u(t',x')-\varepsilon^{-1}|\hat{x}-x'|^2-\delta^{-1}|\hat{t}-t'|^2-\varphi(\hat{t},\hat{x})\\
&=(u^{\varepsilon,\delta}-\varphi)(\hat{t},\hat{x})\\
&\ge u^{\varepsilon,\delta}(t-t'+\hat{t},x-x'+\hat{x})-\varphi(t-t'+\hat{t},x-x'+\hat{x})\\
&\ge u(t,x)-\varepsilon^{-1}|\hat{x}-x'|^2-\delta^{-1}|\hat{t}-t'|^2-\varphi(t-t'+\hat{t},x-x'+\hat{x})
\end{align*}
for $(t,x)$ in a neighborhood of $(t',x')$.
This means that $(t,x)\mapsto u(t,x)-\varphi(t-t'+\hat{t},x-x'+\hat{x})$ attains a local maximum at $(t',x')$.
Since $u$ is a subsolution of \eqref{e:master2} in $(0,T]\times\Omega$, we find that $K_{(0,t')}[u](t',x')$ exists and
\begin{equation}\label{e:conv1.5}
J[u](t',x')+K_{(0,t')}[u](t',x')+F(t',x',u(t',x'),\nabla\varphi(\hat{t},\hat{x}),\nabla^2\varphi(\hat{t},\hat{x}))\le0.
\end{equation}

By a similar observation, we find that $(t,x)\mapsto u^\varepsilon(t,x)-\varphi(t-t'+\hat{t},x)$ attains a local maximum at $(t',\hat{x})$. There is a $\tilde{\varphi}\in C^{1,2}((0,T]\times\Omega)\cap C([0,T]\times\Omega)$ such that $\tilde{\varphi}(t,x)=\varphi(t-t'+\hat{t},x)$ for $(t,x)$ in a sufficiently small neighborhood of $(t',\hat{x})$ and 
\begin{equation}\label{e:conv2}
\max_{[0,T]\times\Omega}(u^\varepsilon-\tilde{\varphi})=(u^\varepsilon-\tilde{\varphi})(t',\hat{x}).
\end{equation} 
It is easy to see that 
\begin{equation}\label{e:conv3}
\tilde{\varphi}(t',\hat{x})-\tilde{\varphi}(t'-\tau,\hat{x})\le u^\varepsilon(t',\hat{x})-u^\varepsilon(t'-\tau,\hat{x})\le u(t',x')-u(t'-\tau,x')
\end{equation}
for all $\tau\in[0,t']$.
Indeed, the left-hand inequality follows from \eqref{e:conv2}, and the right-hand inequality is obtained by combining \eqref{e:conv1} with the inequality $u^\varepsilon(t'-\tau,\hat{x})\ge u(t'-\tau,x')-\varepsilon^{-1}|\hat{x}-x'|^2$.
Since $K_{(0,t')}[\tilde{\varphi}](t',\hat{x})$ and $K_{(0,t')}[u](t',x')$ exist, it turns out that $K_{(0,t')}[u^\varepsilon](t',\hat{x})$ exists.
The right-hand inequality of \eqref{e:conv3} also yields 
\begin{equation}\label{e:conv4}
J[u](t',x')+K_{(0,t')}[u](t',x')\ge J[u^\varepsilon](t',\hat{x})+K_{(0,t')}[u^\varepsilon](t',\hat{x}).
\end{equation}

We next see that
\begin{equation}\label{e:conv5}
J[u^\varepsilon](t',\hat{x})+K_{(0,t')}[u^\varepsilon](t',\hat{x})=\frac{\alpha}{\Gamma(1-\alpha)}\int_0^\infty(u^\varepsilon(t',\hat{x})-\bar{u}^\varepsilon(t'-\tau,\hat{x}))\frac{d\tau}{\tau^{\alpha+1}},
\end{equation}
where the function $\bar{u}^\varepsilon:(-\infty,T]\times\Omega\to\mathbf{R}$ is defined by
\begin{equation*}
	\bar{u}^\varepsilon(t,x):=\begin{cases}
		u^\varepsilon(t,x)\quad&\text{for $(t,x)\in[0,T]\times\Omega$,}\\
		u^\varepsilon(0,x)\quad&\text{for $(t,x)\in(-\infty,0]\times\Omega$.}
	\end{cases}
\end{equation*}
We define $\bar{u}^{\varepsilon,\delta}:(-\infty,T]\times\Omega\to\mathbf{R}$ for $u^{\varepsilon,\delta}$ in the same manner.
In the case of (i) $\tau\ge \hat{t}+M\delta^{1/2}$, (ii) $\tau\le\hat{t}-M\delta^{1/2}$, (iii) $\hat{t}-M\delta^{1/2}\le\tau\le\hat{t}+M\delta^{1/2}$, we evaluate the integrand on the right-hand side of \eqref{e:conv5}.

(i) From \eqref{e:conv} and definitions of $\bar{u}^\varepsilon$ and $u^{\varepsilon,\delta}$ it is seen that
\begin{align*}
u^\varepsilon(t',\hat{x})-\bar{u}^\varepsilon(t'-\tau,\hat{x})
&=u^{\varepsilon,\delta}(\hat{t},\hat{x})+\delta^{-1}|\hat{t}-t'|^2-u^\varepsilon(0,\hat{x})\\
&\ge u^{\varepsilon,\delta}(\hat{t},\hat{x})-u^{\varepsilon,\delta}(0,\hat{x}).
\end{align*} 

(ii) Similarly, it is also seen that
\begin{align*}
u^\varepsilon(t',\hat{x})-\bar{u}^\varepsilon(t'-\tau,\hat{x})
&=u^{\varepsilon,\delta}(\hat{t},\hat{x})-(-\delta^{-1}|\hat{t}-t'|^2+u^\varepsilon(t'-\tau,\hat{x}))\\
&\ge u^{\varepsilon,\delta}(\hat{t},\hat{x})-u^{\varepsilon,\delta}(\hat{t}-\tau,\hat{x}).
\end{align*}

(iii) Since $\sup_{(-\infty,T]\times\Omega}|\bar{u}^{\varepsilon,\delta}|\le\sup_{[0,T]\times\Omega}|u|$, it is seen that
\begin{align*}
&u^\varepsilon(t',\hat{x})-\bar{u}^\varepsilon(t'-\tau,\hat{x})\\
&\ge u^{\varepsilon,\delta}(\hat{t},\hat{x})-\bar{u}^{\varepsilon,\delta}(\hat{t}-\tau,\hat{x})+(\bar{u}^{\varepsilon,\delta}(\hat{t}-\tau,\hat{x})-\bar{u}^\varepsilon(t'-\tau,\hat{x}))\\
&\ge u^{\varepsilon,\delta}(\hat{t},\hat{x})-\bar{u}^{\varepsilon,\delta}(\hat{t}-\tau,\hat{x})-2\sup_{[0,T]\times\Omega}|u|.
\end{align*}
Thus we find that
\begin{align*}
&J[u^\varepsilon](t',\hat{x})+K_{(0,t')}[u^\varepsilon](t',\hat{x})\\
&\ge\frac{\alpha}{\Gamma(1-\alpha)}\left(\int_0^\infty(u^{\varepsilon,\delta}(\hat{t},\hat{x})-\bar{u}^{\varepsilon,\delta}(\hat{t}-\tau,\hat{x}))\frac{d\tau}{\tau^{\alpha+1}}-C^2\int_{\hat{t}-C\delta^{1/2}}^{\hat{t}+C\delta^{1/2}}\frac{d\tau}{\tau^{\alpha+1}}\right)\\
&=J[u^\delta](\hat{t},\hat{x})+K_{(0,\hat{t})}[u^\delta](\hat{t},\hat{x})-\eta_\delta,
\end{align*}
where
$$
\eta_\delta:=C^2\left(\frac{1}{(\hat{t}-M\delta^{1/2})^\alpha}-\frac{1}{(\hat{t}+M\delta^{1/2})^\alpha}\right).
$$

By \eqref{e:conv}, \eqref{e:conv1}, and (A2) it is clear that
$$
F(t',x',u(t',x'),\nabla\varphi(\hat{t},\hat{x}),\nabla^2\varphi(\hat{t},\hat{x}))\ge F_{\varepsilon,\delta}(\hat{t},\hat{x},u^{\varepsilon,\delta}(\hat{t},\hat{x}),\nabla\varphi(\hat{t},\hat{x}),\nabla^2\varphi(\hat{t},\hat{x})).
$$
Therefore we obtain the desired inequality.
\end{proof}

The next proposition is standard.
However, there is no reference which proves the same statement, so we give the detail.

\begin{proposition}\label{p:basic}
Let $u,v:[0,T]\times\Omega\to\mathbf{R}$ be bounded usc and lsc functions, respectively.
Let $\varphi\in C^2([0,T]\times\Omega\times\Omega)$.
Assume that
$$
(t,x,y)\mapsto u^\varepsilon(t,x)-v_\varepsilon(t,y)-\varphi(t,x,y)
$$
attains a strict maximum over $[0,T]\times\overline{\Omega_\varepsilon}\times\overline{\Omega_\varepsilon}$ at $(\bar{t},\bar{x},\bar{y})\in(0,T)\times\Omega_\varepsilon\times\Omega_\varepsilon$.
Let $(t_\delta,x_\delta,s_\delta,y_\delta)$ be a maximum point of
$$
(t,x,s,y)\mapsto u^{\varepsilon,\delta}(t,x)-v_{\varepsilon,\delta}(s,y)-\varphi(t,x,y)-\frac{|t-s|^2}{\delta}
$$
on $([0,T]\times\overline{\Omega_\varepsilon})^2$.
Then 
$$
(t_\delta,x_\delta,s_\delta,y_\delta)\to(\bar{t},\bar{x},\bar{t},\bar{y}),\; u^{\varepsilon,\delta}(t_\delta,x_\delta)\to u^\varepsilon(\bar{t},\bar{x}),\;\text{and}\; v_{\varepsilon,\delta}(s_\delta,y_\delta)\to v_\varepsilon(\bar{t},\bar{y})
$$
 as $\delta\to0$ along a subsequence if necessary.
\end{proposition}

\begin{proof}
Set
\begin{align*}
&\Phi_\varepsilon(t,x,y):=u^\varepsilon(t,x)-v_\varepsilon(t,y)-\varphi(t,x,y)\quad\text{and}\\
&\Phi_{\varepsilon,\delta}(t,x,s,y):=u^{\varepsilon,\delta}(t,x)-v_{\varepsilon,\delta}(s,y)-\varphi(t,x,y)-\frac{|t-s|^2}{\delta}.
\end{align*}
Let $(t_\delta',s_\delta')$ be such that
\begin{equation}\label{e:basic0}
\begin{split}
&u^{\varepsilon,\delta}(t_\delta,x_\delta)=u^\varepsilon(t_\delta',x_\delta)-\delta^{-1}|t_\delta-t_\delta'|^2\quad\text{and}\\
&v_{\varepsilon,\delta}(s_\delta,y_\delta)=v_\varepsilon(s_\delta',y_\delta)+\delta^{-1}|s_\delta-s_\delta'|^2.
\end{split}
\end{equation}
Since $\Phi_{\varepsilon,\delta}(t_\delta,x_\delta,s_\delta,y_\delta)\ge\Phi_{\varepsilon,\delta}(\bar{t},\bar{x},\bar{t},\bar{y})\ge\Phi_\varepsilon(\bar{t},\bar{x},\bar{y})$, we have
\begin{equation}\label{e:basic}
\begin{split}
\frac{|t_\delta-s_\delta|^2}{\delta}
&\le u^{\varepsilon,\delta}(t_\delta,x_\delta)-v_{\varepsilon,\delta}(s_\delta,y_\delta)-\varphi(t_\delta,x_\delta,y_\delta)-\Phi_\varepsilon(\bar{t},\bar{x},\bar{y})\\
&\le u^\varepsilon(t_\delta',x_\delta)-v_\varepsilon(s_\delta',y_\delta)-\varphi(t_\delta,x_\delta,y_\delta)-\Phi_\varepsilon(\bar{t},\bar{x},\bar{y}).
\end{split}
\end{equation}
The last term is bounded from above uniformly in $\delta$.
Hence this inequality implies that $t_\delta$ and $s_\delta$ converge to a same point, say, $\tilde{t}$ as $\delta\to0$ by taking a subsequence if necessary.
Notice that $t_\delta'$ and $s_\delta'$ also converge to $\tilde{t}$ as $\delta\to0$.
Let $(\tilde{x},\tilde{y})\in\overline{\Omega_\varepsilon}\times\overline{\Omega_\varepsilon}$ be such that $(x_\delta,y_\delta)\to(\tilde{x},\tilde{y})$ as $\delta\to0$.
If $(\tilde{t},\tilde{x},\tilde{y})\neq(\bar{t},\bar{x},\bar{y})$, then by \eqref{e:basic} we have
\begin{align*}
0&\le \liminf_{\delta\to0}(u^{\varepsilon,\delta}(t_\delta,x_\delta)-v_{\varepsilon,\delta}(s_\delta,y_\delta)-\varphi(t_\delta,x_\delta,y_\delta)-\Phi_\varepsilon(\bar{t},\bar{x},\bar{y}))\\
&\le \liminf_{\delta\to0}(u^\varepsilon(t_\delta',x_\delta)-v_\varepsilon(s_\delta',y_\delta)-\varphi(t_\delta,x_\delta,y_\delta)-\Phi_\varepsilon(\bar{t},\bar{x},\bar{y}))\\
&\le \limsup_{\delta\to0}(u^\varepsilon(t_\delta',x_\delta)-v_\varepsilon(s_\delta',y_\delta)-\varphi(t_\delta,x_\delta,y_\delta)-\Phi_\varepsilon(\bar{t},\bar{x},\bar{y}))\\
&\le\Phi_\varepsilon(\tilde{t},\tilde{x},\tilde{y})-\Phi_\varepsilon(\bar{t},\bar{x},\bar{y})<0.
\end{align*}
Thus $(t_\delta,x_\delta,s_\delta,y_\delta)\to(\tilde{t},\tilde{x},\tilde{t},\tilde{y})=(\bar{t},\bar{x},\bar{t},\bar{y})$ as $\delta\to0$.
At the same time we see that 
$$
\lim_{\delta\to0}(u^{\varepsilon,\delta}(t_\delta,x_\delta)-v_{\varepsilon,\delta}(s_\delta,y_\delta))=u^\varepsilon(\bar{t},\bar{x})-v_\varepsilon(\bar{t},\bar{y}).
$$
Since $u^\varepsilon$ and $-v_\varepsilon$ are usc, by \eqref{e:basic0} we have
\begin{align*}
0&\ge\limsup_{\delta\to0}(u^{\varepsilon,\delta}(t_\delta,x_\delta)-u^\varepsilon(\bar{t},\bar{x}))\\
&\ge\liminf_{\delta\to0}(u^{\varepsilon,\delta}(t_\delta,x_\delta)-u^\varepsilon(\bar{t},\bar{x}))\\
&\ge\liminf_{\delta\to0}(v_{\varepsilon,\delta}(s_\delta,y_\delta)-v_\varepsilon(\bar{t},\bar{y}))\ge0.
\end{align*}
Therefore $u^{\varepsilon,\delta}(t_\delta,x_\delta)\to u^\varepsilon(\bar{t},\bar{x})$ as $\delta\to0$.
The similar applies to the case for $v_{\varepsilon,\delta}$.
\end{proof}

We are ready to give a proof of Proposition \ref{p:weakishii}.

\begin{proof}[Proof of Proposition \ref{p:weakishii}]
We define a function $\Phi_{\varepsilon,\delta}$ by
\begin{align*}
&\Phi_{\varepsilon,\delta}(t,x,s,y)=u^{\varepsilon,\delta}(t,x)-v_{\varepsilon,\delta}(s,y)-\varphi_\delta(t,x,y)\quad\text{with}\\
&\varphi_\delta(t,x,s,y)=\varphi(t,x,y)+|t-\bar{t}|^2+|x-\bar{x}|^4+|y-\bar{y}|^4+\frac{|t-s|^2}{\delta}.
\end{align*}
Let $z_\delta:=(t_\delta,x_\delta,s_\delta,y_\delta)$ be a maximum point of $\Phi_{\varepsilon,\delta}$ on $([0,T]\times\overline{\Omega_\varepsilon})^2$.
Proposition \ref{p:basic} implies that $z_\delta$ converges to $(\bar{t},\bar{x},\bar{t},\bar{y})$ as $\delta\to0$ (by taking a subsequence if necessary).
Hence $z_\delta\in ((0,T)\times\Omega_\varepsilon)^2$ for sufficiently small $\delta$.
By adding $|t-t_\delta|^2+|x-x_\delta|^4+|s-s_\delta|^2+|y-y_\delta|^4$ to $\varphi_\delta$, we may assume that $\Phi_{\varepsilon,\delta}$ attains a strict maximum at $z_\delta$.
Since $\Phi_{\varepsilon,\delta}$ is also semiconvex, according to \cite[Theorem A.2]{CrandallIshiiLions1992}  there are sequences $z_{\delta,\sigma}:=(t_{\delta,\sigma},x_{\delta,\sigma},s_{\delta,\sigma},y_{\delta,\sigma})\in ((0,T)\times \Omega_\varepsilon)^2$ satisfying $z_{\delta,\sigma}\to z_\delta$ as $\sigma\searrow0$ and $(a_\sigma,\zeta_\sigma,b_\sigma,\xi_\sigma)\in\mathbf{R}^{2(1+d)}$ satisfying $|a_\sigma|+|\zeta_\sigma|+|b_\sigma|+|\xi_\sigma|\le\sigma$ such that the function
\begin{align*}
&\Phi_{\varepsilon,\delta,\sigma}(t,x,s,y):=u^{\varepsilon,\delta}(t,x)-v_{\varepsilon,\delta}(s,y)-\varphi_{\delta,\sigma}(t,x,s,y)\quad\text{with}\\
&\varphi_{\delta,\sigma}(t,x,s,y):=\varphi_\delta(t,x,s,y)+a_\sigma t+\zeta_\sigma\cdot x-b_\sigma s-\xi_\sigma\cdot y
\end{align*}
attains a maximum over $([0,T]\times\overline{\Omega_\varepsilon})^2$ at $z_{\delta,\sigma}$ and has a second differential at $z_{\delta,\sigma}$.
Since
\begin{align*}
\Phi_{\varepsilon,\delta,\sigma}(z_{\delta,\sigma})
&\ge\Phi_{\varepsilon,\delta}(z_\delta)-a_\sigma t_{\delta,\sigma}-\zeta_\sigma\cdot x_{\delta,\sigma}+b_\sigma s_{\delta,\sigma}+\xi_\sigma\cdot y_{\delta,\sigma}\\
&\ge\Phi_\varepsilon(\bar{t},\bar{x},\bar{y})-a_\sigma t_{\delta,\sigma}-\zeta_\sigma\cdot x_{\delta,\sigma}+b_\sigma s_{\delta,\sigma}+\xi_\sigma\cdot y_{\delta,\sigma}
\end{align*}
and since $a_\sigma t_{\delta,\sigma}+\zeta_\sigma\cdot x_{\delta,\sigma}-b_\sigma s_{\delta,\sigma}-\xi_\sigma\cdot y_{\delta,\sigma}$ vanishes as $\sigma\to0$, a slight change in the proof of Proposition \ref{p:basic} shows that 
\begin{equation}\label{e:weakishii_1}
z_{\delta,\sigma}\to(\bar{t},\bar{x},\bar{t},\bar{y}),\; u^{\varepsilon,\delta}(t_{\delta,\sigma},x_{\delta,\sigma})\to u^\varepsilon(\bar{t},\bar{x}),\;\text{and}\; v_{\varepsilon,\delta}(s_{\delta,\sigma},y_{\delta,\sigma})\to v_\varepsilon(\bar{t},\bar{y})
\end{equation}
as $\delta,\sigma\to0$.

We set
\begin{align*}
&p_{\delta,\sigma}:=\nabla_x\varphi_{\delta,\sigma}(t_{\delta,\sigma},x_{\delta,\sigma},y_{\delta,\sigma}),\quad q_{\delta,\sigma}:=-\nabla_y\varphi_{\delta,\sigma}(t_{\delta,\sigma},x_{\delta,\sigma},y_{\delta,\sigma}),\\
&X_{\delta,\sigma}:=\nabla_x^2u^{\varepsilon,\delta}(t_{\delta,\sigma},x_{\delta,\sigma}),\quad\text{and}\quad Y_{\delta,\sigma}:=-\nabla_y^2v_{\varepsilon,\delta}(s_{\delta,\sigma},y_{\delta,\sigma}).
\end{align*}
It is immediate that $(p_{\delta,\sigma},q_{\delta,\sigma})\to(\nabla_x\varphi(\bar{t},\bar{x},\bar{y}),-\nabla_y\varphi(\bar{t},\bar{x},\bar{y}))$ as $\delta,\sigma\to0$.
Differentiating $\Phi_{\varepsilon,\delta,\sigma}$ at $z_{\delta,\sigma}$, we see that
\begin{equation*}
(p_{\delta,\sigma},X_{\delta,\sigma})\in\widetilde{\mathcal{P}}^+u^{\varepsilon,\delta}(t_{\delta,\sigma},x_{\delta,\sigma}),\quad (-q_{\delta,\sigma},-Y_{\delta,\sigma})\in\widetilde{\mathcal{P}}^-v_{\varepsilon,\delta}(s_{\delta,\sigma},y_{\delta,\sigma})
\end{equation*}
and
\begin{equation}\label{e:weakishii_3}
-\frac{2}{\varepsilon}\left(\begin{array}{cc}I&O\\ O&I\end{array}\right)\le\left(\begin{array}{cc}X_{\delta,\sigma}&O\\ O&Y_{\delta,\sigma}\end{array}\right)\le \nabla_{x,y}^2\varphi_{\delta,\sigma}(z_{\delta,\sigma}).
\end{equation}
Since $\nabla_{x,y}^2\varphi_{\delta,\sigma}(z_{\delta,\sigma})$ is bounded from above uniformly in $\delta$ and $\sigma$, by compactness (\cite[Lemma 5.3]{Ishii1989}) there are decreasing sequences $\{\delta_j\}$ and $\{\sigma_j\}$, and matrices $X,Y\in\mathbf{S}^d$ such that $X_j:=X_{\delta_j,\sigma_j}\to X$ and $Y_j:=Y_{\delta_j,\sigma_j}\to Y$ as $j\to\infty$.
Clearly, $\nabla_{x,y}^2\varphi_{\delta_j,\sigma_j}(z_{\delta_j,\sigma_j})\to\nabla_{x,y}^2\varphi(\bar{t},\bar{x},\bar{y})$ as $j\to\infty$.
Therefore we obtain the desired matrix inequality \eqref{e:comparison1} by taking the limit in \eqref{e:weakishii_3} as $j\to\infty$.
Set
$$
(t_j,x_j,s_j,y_j,p_j,q_j):=(t_{\delta_j,\sigma_j},x_{\delta_j,\sigma_j},s_{\delta_j,\sigma_j},y_{\delta_j,\sigma_j},p_{\delta_j,\sigma_j},q_{\delta_j,\sigma_j}).
$$

We notice that $(t_j,x_j),(s_j,y_j)\in  (M\delta_j^{1/2},T)\times\Omega_\varepsilon$ for large $j$.
Since $u$ is a subsolution of \eqref{e:master2} in $(0,T]\times\Omega$, Proposition \ref{p:equivalence} and Lemma \ref{p:conv} imply that
$$
J[u^{\varepsilon,\delta_j}](t_j,x_j)+K_{(0,t_j)}[u^{\varepsilon,\delta_j}](t_j,x_j)+F_{\varepsilon,\delta_j}(t_j,x_j,u^{\varepsilon,\delta_j}(t_j,x_j),p_j,X_j)\le\eta_{\delta_j}^u
$$
for $\eta_{\delta_j}^u$ such that $\eta_{\delta_j}^u\to0$ as $j\to0$.
Similarly, since $v$ is a supersolution of \eqref{e:master2} in $(0,T]\times\Omega$, we have
$$
J[v_{\varepsilon,\delta_j}](s_j,y_j)+K_{(0,s_j)}[v_{\varepsilon,\delta_j}](s_j,y_j)+F^{\varepsilon,\delta_j}(s_j,y_j,v_{\varepsilon,\delta_j}(s_j,y_j),-q_j,-Y_j)\ge\eta_{\delta_j}^v
$$
for $\eta_{\delta_j}^v$ such that $\eta_{\delta_j}^v\to0$ as $j\to\infty$.
Subtracting the second inequality from the first inequality yields
\begin{equation}\label{e:weakishii_4}
\begin{split}
&J[u^{\varepsilon,\delta_j}](t_j,x_j)-J[v_{\varepsilon,\delta_j}](s_j,y_j)+K_{(0,t_j)}[u^{\varepsilon,\delta_j}](t_j,x_j)-K_{(0,s_j)}[v_{\varepsilon,\delta_j}](s_j,y_j)\\
&+F_{\varepsilon,\delta_j}(t_j,x_j,u^{\varepsilon,\delta_j}(t_j,x_j),p_j,X_j)-F^{\varepsilon,\delta_j}(s_j,y_j,v_{\varepsilon,\delta_j}(s_j,y_j),-q_j,-Y_j)\le\eta_{\delta_j}^u-\eta_{\delta_j}^v.
\end{split}
\end{equation}

Since $u^\varepsilon$ and $-v_\varepsilon$ are usc and \eqref{e:weakishii_1} holds, it follows that
\begin{align*}
&\liminf_{j\to\infty}(J[u^{\varepsilon,\delta_j}](t_j,x_j)-J[v_{\varepsilon,\delta_j}](s_j,y_j))\\
&\ge \liminf_{j\to\infty}\left(\frac{u^{\varepsilon,\delta_j}(t_j,x_j)-u^\varepsilon(0,x_j)}{t_j^\alpha\Gamma(1-\alpha)}-\frac{v_{\varepsilon,\delta_j}(s_j,y_j)-v_\varepsilon(0,y_j)}{s_j^\alpha\Gamma(1-\alpha)}\right)\\
&\ge J[u^\varepsilon](\bar{t},\bar{x})-J[v_\varepsilon](\bar{t},\bar{y}).
\end{align*}
There is $(t'_j,x'_j)\in[t_j-M\delta_j^{1/2},t_j+M\delta_j^{1/2}]\times\overline{B(x_j,M\varepsilon^{1/2})}$ such that
$$
F_{\varepsilon,\delta_j}(t_j,x_j,u^{\varepsilon,\delta_j}(t_j,x_j),p_j,X_j)=F(t'_j,x'_j,u^{\varepsilon,\delta_j}(t_j,x_j),p_j,X_j).
$$
Let $\bar{x}'\in\overline{B(\bar{x},M\varepsilon^{1/2})}$ be a limit point of $x'_j$ as $j\to\infty$ along a subsequence.
By (A1) (the continuity of $F$), we find that
\begin{align*}
\lim_{j\to\infty}F(t'_j,x'_j,u^{\varepsilon,\delta_j}(t_j,x_j),p_j,X_j)
&=F(\bar{t},\bar{x}',u^\varepsilon(\bar{t},\bar{x}),\nabla_x\varphi(\bar{t},\bar{x},\bar{y}),X)\\
&\ge F_\varepsilon(\bar{t},\bar{x},u^\varepsilon(\bar{t},\bar{x}),\nabla_x\varphi(\bar{t},\bar{x},\bar{y}),X).
\end{align*}
Similarly, we have
$$
\lim_{j\to\infty}F^{\varepsilon,\delta_j}(s_j,y_j,v_{\varepsilon,\delta_j}(s_j,y_j),-q_j,-Y_j)\le F^\varepsilon(\bar{t},\bar{y},v_\varepsilon(\bar{t},\bar{y}),-\nabla_y\varphi(\bar{t},\bar{x},\bar{y}),-Y).
$$
Let $\rho>0$ be a parameter such that $\rho<t_j,s_j<T-\rho$ for all large $j$.
Since 
$$
(u^{\varepsilon,\delta_j}(t_j,x_j)-u^{\varepsilon,\delta_j}(t_j-\tau,x_j))\mathds{1}_{(\rho,t_j)}(\tau)\ge -2\sup_{[0,T]\times\Omega}|u|\mathds{1}_{(\rho,T)}(\tau)\quad\text{for $\tau\in(0,T)$}
$$
and since the right-hand side of this inequality multiplied by $\tau^{-\alpha-1}$ is integrable on $(0,T)$, Fatou's lemma implies that
$$
\liminf_{j\to\infty}K_{(\rho,t_j)}[u^{\varepsilon,\delta_j}](t_j,x_j)\ge K_{(\rho,\bar{t})}[u^\varepsilon](\bar{t},\bar{x}).
$$
The similar implies that
$$
\limsup_{j\to\infty}K_{(\rho,s_j)}[v_{\varepsilon,\delta_j}](s_j,y_j)\le K_{(\rho,\bar{t})}[v_\varepsilon](\bar{t},\bar{y}).
$$
Since $\Phi_{\varepsilon,\delta_j,\sigma_j}(z_j)\ge\Phi_{\varepsilon,\delta_j,\sigma_j}(t_j-\tau,x_j,s_j-\tau,y_j)$ for $\tau\in[0,\rho]$, we have
$$
K_{(0,\rho)}[u^{\varepsilon,\delta_j}](t_j,x_j)-K_{(0,\rho)}[v_{\varepsilon,\delta_j}](s_j,y_j)\ge K_{(0,\rho)}[\varphi_{\delta_j,\sigma_j}](z_j).
$$
Recall that $\varphi_{\delta,\sigma}$ has a form
\begin{equation}\label{e:weakishii_5}
\begin{split}
\varphi_{\delta,\sigma}(t,x,s,y)
&=\varphi(t,x,y)+|t-\bar{t}|^2+|x-\bar{x}|^4+|y-\bar{y}|^4+\frac{|t-s|^2}{\delta}\\
&+|t-t_\delta|^2+|x-x_\delta|^4+|y-y_\delta|^4+a_\sigma t+\zeta_\sigma\cdot x-b_\sigma s-\xi_\sigma\cdot y
\end{split}
\end{equation}
It is not hard to see that
$$
K_{(0,\rho)}[\varphi_{\delta_j,\sigma_j}](z_j)=K_{(0,\rho)}[\varphi](t_j,x_j,y_j)+C_{\rho,j}
$$
for some constant $C_{\rho,j}$ such that $\lim_{\rho\to0}\lim_{j\to\infty}C_{\rho,j}=0$.
The dominated convergence theorem yields $\lim_{j\to\infty}K_{(0,\rho)}[\varphi](t_j,x_j,y_j)=K_{(0,\rho)}[\varphi](\bar{t},\bar{x},\bar{y})$.

Consequently, by taking the limit infimum in \eqref{e:weakishii_4} as $j\to\infty$, we get
\begin{align*}
&\lim_{j\to\infty}C_{\rho,j}+K_{(0,\rho)}[\varphi](\bar{t},\bar{x},\bar{y})+K_{(\rho,\bar{t})}[u^\varepsilon](\bar{t},\bar{x})-K_{(\rho,\bar{t})}[v_\varepsilon](\bar{t},\bar{y})\\
&+F_\varepsilon(\bar{t},\bar{x},u^\varepsilon(\bar{t},\bar{x}),\nabla_x\varphi(\bar{t},\bar{x},\bar{y}),X)-F^\varepsilon(\bar{t},\bar{y},v_\varepsilon(\bar{t},\bar{y}),-\nabla_y\varphi(\bar{t},\bar{x},\bar{y}),-Y)\le0.
\end{align*}
A similar argument in the proof of Proposition \ref{p:alternative1} ensures that $K_{(0,\bar{t})}[u^\varepsilon](\bar{t},\bar{x})-K_{(0,\bar{t})}[v_\varepsilon](\bar{t},\bar{y})$ exists and the desired inequality \eqref{e:comparison2} holds as $\rho\to0$.
The proof is now complete.
\end{proof}

In order to complete the proof of Lemma \ref{l:ishii} we use an extension to equations with Caputo time fractional derivatives of the Accessibility lemma given by Chen, Giga, and Goto in \cite[Section 2]{ChenGigaGoto19912}.
This is proved with the same idea as their proof. However, we have to take care that the barrier function at $t=T$ must be chosen such that the Caputo time fractional derivative there is large.

\begin{proposition}\label{p:accessibility}
Assume (A1).
Let $u,v:[0,T]\times\Omega\to\mathbf{R}$ be an usc subsolution and a lsc supersolution of \eqref{e:master2} in $(0,T]\times\Omega$, respectively.
Let $(\hat{x},\hat{y})\in\Omega\times\Omega$.
Then there exists a sequence $(t_j,x_j,y_j)\in(0,T)\times\Omega\times\Omega$ such that 
$$
(t_j,x_j,u(t_j,x_j)-v(t_j,y_j))\to(T,\hat{x},u(T,\hat{x})-v(T,\hat{y}))\quad\text{as $j\to\infty$.}
$$
\end{proposition}

\begin{proof}
Suppose that the conclusion were false.
Then there would exist open balls $B(\hat{x})$ and $B(\hat{y})$ centered at $\hat{x}$ and $\hat{y}$, respectively, and $\rho>0$ such that 
$$
a:=u(T,\hat{x})-v(T,\hat{y})-\sup_{(T-\rho, T)\times\overline{B(\hat{x})}\times\overline{B(\hat{y})}}(u(t,x)-v(t,y))>0.
$$
Fix such $B(\hat{x})$, $B(\hat{y})$ and $\rho$.

There is $(\tilde{x},\tilde{y})\in B(\hat{x})\times B(\hat{y})$ such that 
\begin{align*}
&(x,y)\mapsto u(T,x)-v(T,y)-\varphi(x,y)\quad\text{with}\\
&\varphi(x,y):=\frac{|x-\hat{x}|^4+|y-\hat{y}|^4}{\eta}-|x-\tilde{x}|^4-|y-\tilde{y}|^4
\end{align*}
attains a strict maximum over $\overline{B(\hat{x})}\times\overline{B(\hat{y})}$ at $(\tilde{x},\tilde{y})$ by taking a small $\eta>0$.
For a parameter $\lambda\in(0,\rho)$ we define a function $\psi_\lambda$ by
\begin{equation*}
	\psi_\lambda(t)=\begin{cases}
		\frac{a}{2\lambda^\alpha}(t-T+\lambda)^\alpha\quad&\text{for $t\in[T-\lambda,T]$,}\\
		0\quad&\text{for $t\in[0,T-\lambda]$.}
	\end{cases}
\end{equation*}
Set 
$$
\Phi(t,x,s,y)=u(t,x)-v(s,y)-\varphi(x,y)-\psi_\lambda(t)-\frac{|t-s|^2}{\sigma}
$$
for $\sigma>0$. It is easy to see that $\Phi(t,x,t,y)$ attains a strict maximum over $[T-\lambda,T]\times\overline{B(\hat{x})}\times[T-\lambda,T]\times\overline{B(\hat{y})}$ at $(T,\tilde{x},T,\tilde{y})$.
Let $(t_\sigma,x_\sigma,s_\sigma,y_\sigma)$ be a maximum point of $\Phi(t,x,s,y)$ on $[T-\lambda,T]\times\overline{B(\hat{x})}\times[T-\lambda,T]\times\overline{B(\hat{y})}$.
Following the idea of the proof of Proposition \ref{p:basic} it turns out that
\begin{equation}\label{e:accessibility1}
(t_\sigma,x_\sigma,s_\sigma,y_\sigma)\to(T,\tilde{x},T,\tilde{y}),\; u(t_\sigma,x_\sigma)\to u(T,\tilde{x}),\; \text{and}\; v(s_\sigma,y_\sigma)\to v(T,\tilde{y})
\end{equation}
as $\sigma\to0$.

Since $u$ is a subsolution of \eqref{e:master2} in $(0,T]\times\Omega$, we have
$$
J[u](t_\sigma,x_\sigma)+K_{(0,t_\sigma)}[u](t_\sigma,x_\sigma)+F(t_\delta, x_\delta, u(t_\delta, x_\delta),\nabla_x\varphi(x_\delta, y_\delta),\nabla_x^2\varphi(x_\delta, y_\delta))\le0.
$$
Similarly, since $v$ is a supersolution of \eqref{e:master2} in $(0,T]\times\Omega$, we have
$$
J[v](s_\sigma,y_\sigma)+K_{(0,s_\sigma)}[v](s_\sigma,y_\sigma)+F(s_\delta, y_\delta, v(s_\delta, y_\delta),-\nabla_y\varphi(x_\delta, y_\delta),-\nabla_y^2\varphi(x_\delta, y_\delta))\ge0.
$$
Subtracting the second inequality from the the first inequality yields
\begin{equation}\label{e:accessibility2}
J[u](t_\sigma,x_\sigma)-J[v](s_\sigma,y_\sigma)+K_{(0,t_\sigma)}[u](t_\sigma,x_\sigma)-K_{(0,s_\sigma)}[v](s_\sigma,y_\sigma)+\mathcal{F}_\sigma\le0,
\end{equation}
where
\begin{align*}
\mathcal{F}(\delta)&=F(t_\delta, x_\delta, u(t_\delta, x_\delta),\nabla_x\varphi(x_\delta, y_\delta),\nabla_x^2\varphi(x_\delta, y_\delta))\\
&\quad-F(s_\delta, y_\delta, v(s_\delta, y_\delta),-\nabla_y\varphi(x_\delta, y_\delta),-\nabla_y^2\varphi(x_\delta, y_\delta)).
\end{align*}
By \eqref{e:accessibility1} and the upper semicontinuity of $u$ and $-v$ we see that 
$$
\liminf_{\delta\to0}(J[u](t_\sigma,x_\sigma)-J[v](s_\sigma,y_\sigma))\ge(J[u](T,\tilde{x})-J[v](T,\tilde{x})).
$$
It is immediate that
\begin{align*}
\lim_{\sigma\to0}\mathcal{F}_\sigma&=F(T,\tilde{x},u(T,\tilde{x}),\nabla_x\varphi(\tilde{x},\tilde{y}),\nabla_x^2\varphi(\tilde{x},\tilde{y}))\\
&-F(T,\tilde{y},v(T,\tilde{y}),-\nabla_y\varphi(\tilde{x},\tilde{y}),-\nabla_y^2\varphi(\tilde{x},\tilde{y}))).
\end{align*}
We denote the right-hand side by $\mathcal{F}$.
Fix $\lambda'\in(0,\lambda)$ arbitrarily.
Since 
$$
\Phi(t_\sigma,x_\sigma,s_\sigma,y_\sigma)\ge\Phi(t_\sigma-\tau,x_\sigma,s_\sigma-\tau,y_\sigma)\quad\text{for $\tau\in[0,\lambda']$}
$$
and sufficiently small $\sigma$, we see that 
\begin{align*}
&K_{(0,t_\sigma)}[u](t_\sigma,x_\sigma)-K_{(0,s_\sigma)}[v](s_\sigma,y_\sigma)\\
&\ge K_{(\lambda',t_\sigma)}[u](t_\sigma,x_\sigma)-K_{(\lambda',s_\sigma)}[v](s_\sigma,y_\sigma)+K_{(0,\lambda')}[\psi_\lambda](t_\sigma).
\end{align*}
Thus Fatou's lemma yields  
\begin{align*}
&\liminf_{\sigma\to0}(K_{(0,t_\sigma)}[u](t_\sigma,x_\sigma)-K_{(0,s_\sigma)}[v](s_\sigma,y_\sigma))\\
&\ge K_{(\lambda',T)}[u](T,\tilde{x})-K_{(\lambda',T)}[v](T,\tilde{y})+K_{(0,\lambda')}[\psi_\lambda](T).
\end{align*}
Set $w(t,x,y)=u(t,x)-v(t,y)$.
Since $\Phi(T,\tilde{x},T,\tilde{y})\ge\Phi(T-\tau,\tilde{x},T-\tau,\tilde{y})$ for $\tau\in[0,\lambda]$, we have
$$
K_{(\lambda',T)}[w](T,\tilde{x},\tilde{y})+K_{(0,\lambda')}[\psi_\lambda](T)\ge K_{(\lambda,T)}[w](T,\tilde{x},\tilde{y})+K_{(0,\lambda)}[\psi_\lambda](T).
$$
Therefore, from \eqref{e:accessibility2}, we find that
\begin{equation*}
J[w](T,\tilde{x},\tilde{y})+K_{(0,\lambda)}[\psi_\lambda](T)+K_{(\lambda,T)}[w](T,\tilde{x},\tilde{y})+\mathcal{F}\le0.
\end{equation*}
We rewrite this as
\begin{equation}\label{e:accessibility3}
K_{(0,\lambda)}[\psi_\lambda](T)+K_{(\lambda,T)}[w^+](T,\tilde{x},\tilde{y})\le K_{(\lambda,T)}[w^-](T,\tilde{x},\tilde{y})-C,
\end{equation}
where $C:=-J[w](T,\tilde{x},\tilde{y})-\mathcal{F}$.

It is well known \cite[Equation (2.56)]{Podlubny1999} that
$$
\frac{1}{\Gamma(1-\alpha)}\int_a^t\frac{\frac{d}{ds}[(s-a)^\beta]}{(t-s)^\alpha}ds=\frac{\Gamma(\beta+1)}{\Gamma(\beta-\alpha+1)}(t-a)^{\beta-\alpha}
$$
for $t,a\in\mathbf{R}$ with $a<t$ and $\beta>0$.
By this and elementary calculations we see that
\begin{align*}
K_{(0,\lambda)}[\psi_\lambda](T)
&=\frac{1}{\Gamma(1-\alpha)}\left(\int_{T-\lambda}^T\frac{\partial_t\psi_\lambda(s)}{(t-s)^\alpha}ds-\frac{\psi_\lambda(T)-\psi_\lambda(T-\lambda)}{\lambda^\alpha}\right)\\
&=\frac{a}{2\lambda^\alpha}\left(\Gamma(1+\alpha)-\frac{1}{\Gamma(1-\alpha)}\right).
\end{align*}
Since $\alpha\in(0,1)$, there is a $c_\alpha>0$ such that $\Gamma(1+\alpha)-\frac{1}{\Gamma(1-\alpha)}\ge c_\alpha$.
This can be verified using $x\Gamma(x)=\Gamma(1+x)$ for all $x\in\mathbf{R}$ except the non-positive integers and Euler's reflection formula, i.e., $\Gamma(x)\Gamma(1-x)=\pi/\sin(\pi x)$ for all $x\in\mathbf{R}\setminus\mathbf{Z}$.
Since $K_{(\lambda,T)}[w^-](T,\tilde{x},\tilde{y})\le K_{(\rho,T)}[w^-](T,\tilde{x},\tilde{y})$, by \eqref{e:accessibility3} we have
\begin{equation*}
\frac{ac_\alpha}{2\lambda^\alpha}\le K_{(\rho,T)}[w^-](T,\tilde{x},\tilde{y})-C,
\end{equation*}
which is contradictory for small $\lambda$.
\end{proof}

\begin{proof}[Proof of Lemma \ref{l:ishii}]
We may assume that $\bar{t}=T$ by Proposition \ref{p:weakishii}.
For $\sigma>0$ we consider
$$
\Phi_{\varepsilon,\sigma}(t,x,y)=u^\varepsilon(t,x)-v_\varepsilon(t,y)-\varphi(t,x,y)-\frac{\sigma}{T-t}
$$
on $[0,T)\times\overline{\Omega_\varepsilon}\times\overline{\Omega_\varepsilon}$.
There is its maximum point $(t_\sigma,x_\sigma,y_\sigma)\in[0,T)\times\overline{\Omega_\varepsilon}\times\overline{\Omega_\varepsilon}$.
By using Proposition \ref{p:accessibility} as in \cite[Section 6]{ChenGigaGoto19912} it follows that
\begin{equation}\label{e:ishiicomplete}
(t_\sigma,x_\sigma,y_\sigma)\to(T,\bar{x},\bar{y}),\quad u^\varepsilon(t_\sigma,x_\sigma)\to u^\varepsilon(T,\bar{x}),\quad\text{and}\quad v_\varepsilon(t_\sigma,y_\sigma)\to v_\varepsilon(T,\bar{y})
\end{equation}
as $\sigma\to0$.
Hence $(t_\sigma,x_\sigma,y_\sigma)\in(0,T)\times\Omega_\varepsilon\times\Omega_\varepsilon$ for sufficiently small $\sigma$.

We are now able to apply Proposition \ref{p:weakishii} in which $(\bar{t},\bar{x},\bar{y})$ and $T$ are respectively replaced with $(t_\sigma,x_\sigma,y_\sigma)$ and some $T_\sigma$ such that $t_\sigma<T_\sigma<T$ for each $\sigma$.
Thus there are two matrices $X_\sigma,Y_\sigma\in\mathbf{S}^d$ satisfying
\begin{equation}\label{e:ishiicomplete2}
-\frac{2}{\varepsilon}\left(\begin{array}{cc}I&O\\O&I\end{array}\right)\le\left(\begin{array}{cc}X_\sigma&O\\ O& Y_\sigma\end{array}\right)\le\nabla_{x,y}^2\varphi(t_\sigma,x_\sigma,y_\sigma)
\end{equation}
such that $K_{(0,t_\sigma)}[u^\varepsilon](t_\sigma,x_\sigma)-K_{(0,t_\sigma)}[v_\varepsilon](t_\sigma,y_\sigma)$ exists and 
\begin{equation}\label{e:ishiicomplete3}
J[u^\varepsilon](t_\sigma,x_\sigma)-J[v_\varepsilon](t_\sigma,y_\sigma)+K_{(0,t_\sigma)}[u^\varepsilon](t_\sigma,x_\sigma)-K_{(0,t_\sigma)}[v_\varepsilon](t_\sigma,y_\sigma)+\mathcal{F}_\sigma\le0,
\end{equation}
where
\begin{align*}
\mathcal{F}_\sigma:&=F_\varepsilon(t_\sigma,x_\sigma,u^\varepsilon(t_\sigma,x_\sigma),\nabla_x\varphi(t_\sigma,x_\sigma,y_\sigma),X_\sigma)\\
&-F^\varepsilon(t_\sigma,y_\sigma,v_\varepsilon(t_\sigma,y_\sigma),-\nabla_y\varphi(t_\sigma,x_\sigma,y_\sigma),-Y_\sigma).
\end{align*}
Since $\nabla_{x,y}^2\varphi(t_\sigma,x_\sigma,y_\sigma)$ is bounded from above uniformly in $\sigma$, \eqref{e:ishiicomplete2} implies that there exist $X,Y\in\mathbf{S}^d$ such that $X_\sigma\to X$ and $Y_\sigma\to Y$ as $\sigma\to0$ by taking a subsequence.
The matrix inequality \eqref{e:comparison1} is obtained by letting $\sigma\to0$ in \eqref{e:ishiicomplete2}.
Let $\rho>0$ be a sufficiently small parameter.
In much the same way as in the proof of Proposition \ref{p:weakishii} we easily see that
\begin{align*}
&\liminf_{\sigma\to0}(J[u^\varepsilon](t_\sigma,x_\sigma)-J[v_\varepsilon](t_\sigma,y_\sigma))\ge J[u^\varepsilon](T,\bar{x})-J[v_\varepsilon](T,\bar{y}),\\
&\liminf_{\sigma\to0}(K_{(\rho,t_\sigma)}[u^\varepsilon](t_\sigma,x_\sigma)-K_{(\rho,t_\sigma)}[v_\varepsilon](t_\sigma,y_\sigma))\\
&\ge K_{(\rho,T)}[u^\varepsilon](T,\bar{x})-K_{(\rho,T)}[v_\varepsilon](T,\bar{y}),
\end{align*}
and
\begin{align*}
\liminf_{\sigma\to0}\mathcal{F}_\sigma\ge& F_\varepsilon(T,\bar{x},u^\varepsilon(T,\bar{x}),\nabla_x\varphi(T,\bar{x},\bar{y}),X)\\
&-F^\varepsilon(T,\bar{y},v_\varepsilon(T,\bar{y}),-\nabla_y\varphi(T,\bar{x},\bar{y}),-Y).
\end{align*}
Since $\Phi_{\varepsilon,\sigma}(t_\sigma,x_\sigma,y_\sigma)\ge\Phi_{\varepsilon,\sigma}(t_\sigma-\tau,x_\sigma,y_\sigma)$ and $\frac{\sigma}{T-t_\sigma}-\frac{\sigma}{T-t_\sigma+\tau}\ge0$ for all $\tau\in[0,t_\sigma]$, we have
$$
K_{(0,\rho)}[u^\varepsilon](t_\sigma,x_\sigma)-K_{(0,\rho)}[v_\varepsilon](t_\sigma,y_\sigma)\ge K_{(0,\rho)}[\varphi](t_\sigma,x_\sigma,y_\sigma).
$$
The dominated convergence theorem implies that 
$$
\lim_{\sigma\to0}K_{(0,\rho)}[\varphi](t_\sigma,x_\sigma,y_\sigma)=K_{(0,\rho)}[\varphi](T,\bar{x},\bar{y}).
$$
Therefore, by taking the limit infimum in \eqref{e:ishiicomplete3} as $\sigma\to0$, we get
\begin{align*}
&K_{(0,\rho)}[\varphi](T,\bar{x},\bar{y})+K_{(\rho,T)}[u^\varepsilon](T,\bar{x})-K_{(\rho,T)}[v_\varepsilon](T,\bar{y})\\
&+F_\varepsilon(T,\bar{x},u^\varepsilon(T,\bar{x}),\nabla_x\varphi(T,\bar{x},\bar{y}),X)-F^\varepsilon(T,\bar{y},v_\varepsilon(T,\bar{y}),-\nabla_y\varphi(T,\bar{x},\bar{y}),-Y)\le0.
\end{align*}
As in the proof of Proposition \ref{p:weakishii} the inequality \eqref{e:comparison2} is obtained by letting $\rho\to0$.
\end{proof}

\section{Existence and uniqueness for the Cauchy-Neumann problem}
We consider the Cauchy-Neumann problem of the form
\begin{equation}\label{e:neumann}
\begin{cases}
\partial_t^\alpha u+F(t,x,u,\nabla u,\nabla^2 u)=0\quad&\text{in $(0,T]\times\Omega$,}\\
\nabla u\cdot n(x)=0\quad&\text{on $(0,T]\times\partial\Omega$}
\end{cases}
\end{equation}
and
\begin{equation}\label{e:initial}
u|_{t=0}=u_0\quad\text{on $\overline{\Omega}$.}
\end{equation}
Here $\Omega$ is a bounded $C^1$ domain in $\mathbf{R}^d$ and $n:\partial\Omega\to\mathbf{R}$ is the outward unit normal.
In this section, we introduce a viscosity solution when boundary conditions are interpreted in the viscosity sense and then investigate a unique existence of a solution for \eqref{e:neumann}-\eqref{e:initial}.
For the latter purpose the following assumptions are made:
\begin{itemize}
\item[(A5)] $\Omega$ satisfies the uniform exterior sphere condition, i.e., there is $r>0$ such that $\overline{B(x+rn(x),r)}\cap\Omega=\emptyset$ for $x\in\partial\Omega$,
\item[(A6)] $F\in C((0,T]\times\overline{\Omega}\times\mathbf{R}\times\mathbf{R}^d\times\mathbf{S}^d)$,
\item[(A7)] for all $(t,x,p,X)\in(0,T]\times\overline{\Omega}\times\mathbf{R}^d\times\mathbf{S}^d$
$$
F(t,x,w_1,p,X)\le F(t,x,w_2,p,X)\quad\text{if $w_1\le w_2$}
$$
\item[(A8)] there is a function $\omega_1:[0,\infty]\to[0,\infty]$ that satisfies $\lim_{r\searrow0}\omega_1(r)=0$ such that
\begin{align*}
&F(t,x,w,\sigma^{-1}(x-y),-Y)-F(t,y,w,\sigma^{-1}(x-y),X)\\
&\le\omega_1(|x-y|(1+\sigma^{-1}|x-y|))
\end{align*}
for all $(t,x,y,w)\in(0,T]\times\overline{\Omega}\times\overline{\Omega}\times\mathbf{R}$, $X,Y\in\mathbf{S}^d$, and $\sigma>0$ satisfying
\begin{equation*}
\left(\begin{array}{cc}X&O\\ O&Y\end{array}\right)\le\frac{2}{\sigma}\left(\begin{array}{cc}I & -I\\ -I&I\end{array}\right).
\end{equation*}
\item[(A9)] there is a neighborhood $V$ of $\partial\Omega$ relative to $\overline{\Omega}$ such that
$$
|F(t,x,u,p,X)-F(t,x,u,q,Y)|\le\omega_2(|p-q|+\|X-Y\|)
$$
for $x\in V$, $p,q\in\mathbf{R}^d$, $X,Y\in\mathbf{S}^d$,
\item[(A10)] $u_0\in C(\overline{\Omega})$
\end{itemize}

\subsection{Boundary condition in the viscosity sense}
Definition \ref{d:solution} of the viscosity solution is extended for more general equations of the form
\begin{equation}\label{e:boundary}
E(t,x,u,\partial_t^\alpha u,\nabla u,\nabla^2 u)=0\quad\text{in $(a,T]\times O$.}
\end{equation}
Here $0\le a<T$, $O$ is a locally compact subset of $\mathbf{R}^d$, and $E$ is a real-valued function on $W:=(a,T]\times O\times\mathbf{R}\times\mathbf{R}\times\mathbf{R}^d\times\mathbf{S}^d$ that satisfies $-\infty<E_*\le E^*<+\infty$ in $W$.

\begin{definition}\label{d:viscosityboundary}
(i) A function $u:[0,T]\times O\to\mathbf{R}$ is a (viscosity) subsolution of \eqref{e:boundary} in $(a,T]\times O$ if $u^*<+\infty$ in $[0,T]\times O$ and
$$
E_*(\hat{t},\hat{x},u^*(\hat{t},\hat{x}),J[\varphi](\hat{t},\hat{x})+K_{(0,\hat{t})}[\varphi](\hat{t},\hat{x})
,\nabla\varphi(\hat{t},\hat{x}),\nabla^2\varphi(\hat{t},\hat{x}))\le0
$$ 
whenever $((\hat{t},\hat{x}),\varphi)\in((a,T]\times O)\times (C^{1,2}((a,T]\times O)\cap C([0,T]\times O))$ satisfies
$$
\max_{[0,T]\times O}(u^*-\varphi)=(u^*-\varphi)(\hat{t},\hat{x}).
$$

(ii) A function $u:[0,T]\times O\to\mathbf{R}$ is a (viscosity) supersolution of \eqref{e:boundary} in $(a,T]\times O$ if $u_*>-\infty$ in $[0,T]\times O$ and
$$
E^*(\hat{t},\hat{x},u_*(\hat{t},\hat{x}),J[\varphi](\hat{t},\hat{x})+K_{(0,\hat{t})}[\varphi](\hat{t},\hat{x})
,\nabla\varphi(\hat{t},\hat{x}),\nabla^2\varphi(\hat{t},\hat{x}))\le0
$$ 
whenever $((\hat{t},\hat{x}),\varphi)\in((a,T]\times O)\times (C^{1,2}((a,T]\times O)\cap C([0,T]\times O))$ satisfies
$$
\min_{[0,T]\times O}(u_*-\varphi)=(u_*-\varphi)(\hat{t},\hat{x}).
$$

(iii) If a function $u:[0,T]\times O\to\mathbf{R}$ is both a (viscosity) sub- and supersolution of \eqref{e:boundary} in $(a,T]\times O$, then $u$ is called a (viscosity) solution of \eqref{e:boundary} in $(a,T]\times O$.
\end{definition}

\begin{proposition}\label{p:equivboundary}
Assume that for all $(t,x,w,p,X)\in(a,T]\times O\times\mathbf{R}\times\mathbf{R}^d\times\mathbf{S}^d$
$$
E_*(t,x,w,l_1,p,X)\le E_*(t,x,w,l_2,p,X)\quad\text{if $l_1\le l_2$.}
$$
Then a function $u:[0,T]\times O\to\mathbf{R}$ is a subsolution of \eqref{e:boundary} in $(a,T]\times O$ if and only if

(i) $u^*<+\infty$ in $[0,T]\times O$, $K_{(0,\hat{t})}[u^*](\hat{t},\hat{x})$ exists, and
$$
E_*(\hat{t},\hat{x},u^*(\hat{t},\hat{x}),J[u^*](\hat{t},\hat{x})+K_{(0,\hat{t})}[u^*](\hat{t},\hat{x})
,\nabla\varphi(\hat{t},\hat{x}),\nabla^2\varphi(\hat{t},\hat{x}))\le0
$$
whenever $((\hat{t},\hat{x}),\varphi)\in((a,T]\times O)\times C^{1,2}((a,T]\times O)$ satisfies
$$
\max_{(a,T]\times O}(u^*-\varphi)=(u^*-\varphi)(\hat{t},\hat{x}),
$$

\noindent or

(ii) $u^*<+\infty$ in $[0,T]\times O$, $K_{(0,\hat{t})}[u^*](\hat{t},\hat{x})$ exists, and
$$
E_*(t,x,u^*(t,x),J[u^*](t,x)+K_{(0,t)}[u^*](t,x),p,X)\le0
$$
for all $(t,x)\in(a,T]\times O$ and $(p,X)\in\widetilde{\mathcal{P}}^+u^*(t,x)$.

The symmetric statement holds for supersolutions.
\end{proposition}

The proof of this proposition parallels those of Propositions \ref{p:alternative1} and \ref{p:equivalence}, so we do not repeat it.

A viscosity solution of \eqref{e:neumann}-\eqref{e:initial} when the boundary condition is interpreted in the viscosity sense is defined by means of Definition \ref{d:viscosityboundary}.

\begin{definition}
A function $u:[0,T]\times\overline{\Omega}\to\mathbf{R}$ is a (viscosity) subsolution of \eqref{e:neumann} if it is a (viscosity) subsolution of \eqref{e:boundary} in $(0,T]\times\overline{\Omega}$ with
\begin{equation}\label{e:definitione}
  E(t,x,w,l,p,X)=\begin{cases}
    l+F(t,x,w,p,X)\quad&\text{if $x\in\overline{\Omega}$,}\\
    p\cdot n(x)\quad&\text{if $x\in\partial\Omega$.}
  \end{cases}
\end{equation}
If a (viscosity) subsolution $u$ of \eqref{e:neumann} satisfies $u^*(0,\cdot)\le u_0$ on $\overline{\Omega}$, then $u$ is called a (viscosity) subsolution of \eqref{e:neumann}-\eqref{e:initial}.

A (viscosity) supersolution and a (viscosity) solution are defined in a similar way.
\end{definition}

This definition is used in this section. Note that
\begin{equation*}
  E_*(t,x,w,l,p,X)=\begin{cases}
    l+F(t,x,w,p,X)\quad&\text{if $x\in\overline{\Omega}$,}\\
    (l+F(t,x,w,p,X))\wedge (p\cdot n(x))\quad&\text{if $x\in\partial\Omega$}
  \end{cases}
\end{equation*}
and
\begin{equation*}
  E^*(t,x,w,l,p,X)=\begin{cases}
    l+F(t,x,w,p,X)\quad&\text{if $x\in\overline{\Omega}$,}\\
    (l+F(t,x,w,p,X))\vee (p\cdot n(x))\quad&\text{if $x\in\partial\Omega$.}
  \end{cases}
\end{equation*}

Of course, a viscosity solution of \eqref{e:neumann} with more general boundary condition can be defined.
Let $B=B(t,x,w,p):(0,T]\times\overline{\Omega}\times\mathbf{R}\times\mathbf{R}^d\to\mathbf{R}$ be a continuous function.
A (viscosity) subsolution of \eqref{e:neumann} with the boundary condition $B$ is defined in the same way by replacing $p\cdot n(x)$ with $B(t,x,w,p)$ in \eqref{e:definitione}.
Definitions of a (viscosity) supersolution and a (viscosity) solution are similar.

\begin{remark}
As in the case of $\alpha=1$, if $\rho\mapsto B(t,x,w,p-\rho n(x))$ is nonincreasing in $\rho\ge0$ then a classical solution of \eqref{e:neumann} with the boundary condition $B$ is a viscosity solution of \eqref{e:neumann} with the boundary condition $B$ in $(0,T]\times O$. See, e.g., \cite[Section 7]{CrandallIshiiLions1992} and \cite[Proposition 2.3.3]{Giga2006} for the proof.
\end{remark}

\subsection{Unique existence of a solution}
Let $u,v:[0,T]\times\overline{\Omega}\to\mathbf{R}$ be bounded functions and $\varepsilon>0$ be a parameter.
In what follows, by $u^\varepsilon$ and $v_\varepsilon$, we denote the sup- and inf-convolution in space of $u$ and $v$ defined by
\begin{align*}
&u^\varepsilon(t,x)=\sup_{x'\in\overline{\Omega}}\{u(t,x')-\varepsilon^{-1}|x-x'|^2\}\quad\text{and}\\
&v_\varepsilon(t,x)=\inf_{x'\in\overline{\Omega}}\{v(t,x')+\varepsilon^{-1}|x-x'|^2\},
\end{align*}
respectively.
Let $C$ be such that $C>(2(\sup_{[0,T]\times\overline{\Omega}}|u|+\sup_{[0,T]\times\overline{\Omega}}|v|))^{1/2}$.
Set $\Omega_\varepsilon:=\{x\in\mathbf{R}^d\mid \dist(x,\overline{\Omega})<C\varepsilon^{1/2}\}$, $M:=C+(2\sup_{[0,T]\times\overline{\Omega}}|u|)^{1/2}\vee(2\sup_{[0,T]\times\overline{\Omega}}|v|)^{1/2}$
, and $B_{x,\varepsilon}:=\overline{\Omega}\cap\overline{B(x,M\varepsilon^{1/2})}$.
From \cite{IshiiLions1990} we recall that for $(t,x)\in[0,T]\times\Omega_\varepsilon$ there exist $x',y'\in B_{x,\varepsilon}$ such that
$$
u^\varepsilon(t,x)=u(t,x')-\varepsilon^{-1}|x-x'|^2\quad\text{and}\quad v_\varepsilon(t,x)=v(t,x')+\varepsilon^{-1}|x-x'|^2.
$$
For a parameter $\delta>0$ let $u^{\varepsilon,\delta}$ and $v_{\varepsilon,\delta}$ denote the sup- and inf-convolution in time of $u^\varepsilon$ and $v_\varepsilon$ defined as before, respectively.

\begin{proposition}\label{p:convboundary}
Let $u,v:[0,T]\times\overline{\Omega}\to\mathbf{R}$ be a bounded usc subsolution and a bounded lsc supersolution of \eqref{e:boundary} in $(0,T]\times\overline{\Omega}$.
Assume that for all $(t,x,w,p,X)\in(0,T]\times\overline{\Omega}\times\mathbf{R}\times\mathbf{R}^d\times\mathbf{S}^d$
$$
E_*(t,x,w,l_1,p,X)\le E_*(t,x,w,l_2,p,X)
$$
and
$$
E^*(t,x,w,l_1,p,X)\le E^*(t,x,w,l_2,p,X)
$$
if $l_1\le l_2$.
Then $u^{\varepsilon,\delta}$ and $v_{\varepsilon,\delta}$ are a subsolution and a supersolution of
$$
E_{\varepsilon,\delta}(t,x,u^{\varepsilon,\delta},\partial_t^\alpha u^{\varepsilon,\delta}-\eta_\delta^u,\nabla u^{\varepsilon,\delta},\nabla^2 u^{\varepsilon,\delta})=0\quad\text{in $(M\delta^{1/2},T]\times\Omega_\varepsilon$}
$$
and
$$
E^{\varepsilon,\delta}(t,x,v_{\varepsilon,\delta},\partial_t^\alpha v_{\varepsilon,\delta}+\eta_\delta^v,\nabla v_{\varepsilon,\delta},\nabla^2 v_{\varepsilon,\delta})=0\quad\text{in $(M\delta^{1/2},T]\times\Omega_\varepsilon$,}
$$
respectively.
Here
\begin{align*}
&E_{\varepsilon,\delta}(t,x,w,l,p,X)=\min\{E_*(t',x',w,l,p,X) \mid |t-t'|\le M\delta^{1/2},x'\in B_{x,\varepsilon}\},\\
&E^{\varepsilon,\delta}(t,x,w,l,p,X)=\max\{E^*(t',x',w,l,p,X) \mid |t-t'|\le M\delta^{1/2},x'\in B_{x,\varepsilon}\},
\end{align*}
and $\eta_\delta^u,\eta_\delta^v$ are constants such that $\eta_\delta^u,\eta_\delta^v\to0$ as $\delta\to0$. 
\end{proposition}

The proof of this proposition is similar to that of Proposition \ref{p:conv}, so it is safely left to the reader.

\begin{proposition}\label{p:boundaryishii}
Assume (A5).
For a fixed constant $\rho>0$ let $u:[0,T]\times\Omega\to\mathbf{R}$ be a bounded usc subsolution of 
\begin{equation}\label{e:boundaryishii1}
  \begin{cases}
    \partial_t^\alpha u+F(t,x,u,\nabla u,\nabla^2 u)=0\quad&\text{in $(0,T]\times\Omega$,}\\
    \nabla u\cdot n(x)+\rho=0\quad&\text{on $[0,T]\times\partial\Omega$}
  \end{cases}
\end{equation}
and $v:[0,T]\times\overline{\Omega}\to\mathbf{R}$ be a bounded lsc supersolution of
\begin{equation}\label{e:boundaryishii2}
  \begin{cases}
    \partial_t^\alpha v+G(t,x,v,\nabla v,\nabla^2 v)=0\quad&\text{in $(0,T]\times\Omega$,}\\
    \nabla v\cdot n(x)-\rho=0\quad&\text{on $[0,T]\times\partial\Omega$.}
  \end{cases}
\end{equation}
Here $F,G:(0,T]\times\overline{\Omega}\times\mathbf{R}\times\mathbf{R}^d\times\mathbf{R}^d\times\mathbf{S}^d\to\mathbf{R}$ are functions that satisfy (A6) and (A7).
Let $\varphi\in C^2([0,T]\times\mathbf{R}^d\times\mathbf{R}^d)$.
Let $(\bar{t},\bar{x},\bar{y})\in(0,T]\times\Omega_\varepsilon\times\Omega_\varepsilon$ be such that
$$
\max_{(t,x,y)\in[0,T]\times\overline{\Omega_\varepsilon}\times\overline{\Omega_\varepsilon}}(u^\varepsilon(t,x)-v_\varepsilon(t,y)-\varphi(t,x,y))=u^\varepsilon(\bar{t},\bar{x})-v_\varepsilon(\bar{t},\bar{y})-\varphi(\bar{t},\bar{x},\bar{y}).
$$
Then there exist two matrices $X,Y\in\mathbf{S}^d$ satisfying 
\begin{equation}\label{e:mat}
-\frac{2}{\varepsilon}\left(\begin{array}{cc}I&O\\ O&I\end{array}\right)\le\left(\begin{array}{cc}X&O\\ O&Y\end{array}\right)\le\nabla_{x,y}^2\varphi(\bar{t},\bar{x},\bar{y})
\end{equation}
such that $K_{(0,\bar{t})}[u^\varepsilon](\bar{t},\bar{x})-K_{(0,\bar{t})}[v_\varepsilon](\bar{t},\bar{y})$ exists and 
\begin{equation}\label{e:ine}
\begin{split}
&J[u^\varepsilon](\bar{t},\bar{x})-J[v_\varepsilon](\bar{t},\bar{y})+K_{(0,\bar{t})}[u^\varepsilon](\bar{t},\bar{x})-K_{(0,\bar{t})}[v_\varepsilon](\bar{t},\bar{y})\\
&+F_\varepsilon(\bar{t},\bar{x},u^\varepsilon(\bar{t},\bar{x}),\nabla_x\varphi(\bar{t},\bar{x},\bar{y}),X)-G^\varepsilon(\bar{t},\bar{y},v_\varepsilon(\bar{t},\bar{y}),-\nabla_y\varphi(\bar{t},\bar{x},\bar{y}),-Y)\le0.
\end{split}
\end{equation}
Here 
\begin{align*}
&F_\varepsilon(t,x,w,p,X)=\min\{F(t,x',w,p,X)\mid x'\in B_{x,\varepsilon}\}\quad\text{and}\\
&G^\varepsilon(t,x,w,p,X)=\max\{G(t,x',w,p,X)\mid x'\in B_{x,\varepsilon}\}.
\end{align*}
\end{proposition}

\begin{proof}
We only give a proof in the case $\bar{t}<T$; the same idea as that of Lemma \ref{l:ishii} applies to the case $\bar{t}=T$.

By a similar argument as in the proof of Proposition \ref{p:weakishii} we have $(t_j,x_j),(s_j,y_j)\in (M\delta_j^{1/2},T)\times\Omega_\varepsilon$, $(p_j,X_j)\in\widetilde{\mathcal{P}}^+u^{\varepsilon,\delta_j}(t_j,x_j)$, and $(-q_j,-Y_j)\in\widetilde{\mathcal{P}}^-v_{\varepsilon,\delta_j}(s_j,y_j)$ such that
$$
(t_j,x_j,s_j,y_j)\to(\bar{t},\bar{x},\bar{t},\bar{y}),\quad(p_j,q_j)\to(\nabla_x\varphi(\bar{t},\bar{x},\bar{y}),-\nabla_y\varphi(\bar{t},\bar{x},\bar{y}))
$$
and $(X_j,Y_j)\to(X,Y)$ as $j\to\infty$ for some $X,Y\in\mathbf{S}^d$ satisfying \eqref{e:mat}.
Note that $p_j$ and $q_j$ are, respectively, represented as $\nabla_x\varphi_{\delta_j,\sigma_j}(t_j,x_j,s_j,y_j)$ and $-\nabla_y\varphi_{\delta_j,\sigma_j}(t_j,x_j,s_j,y_j)$ for $\varphi_{\delta,\sigma}$ with the same form as \eqref{e:weakishii_5}.

Since $u$ is a subsolution of \eqref{e:boundaryishii1}, we have
\begin{equation}\label{e:ishiiboundary3}
E_{\varepsilon,\delta_j}(t_j,x_j,u^{\varepsilon,\delta_j},J[u^{\varepsilon,\delta}](t_j,x_j)+K_{(0,t_j)}[u^{\varepsilon,\delta}](t_j,x_j),p_j,X_j)\le0,
\end{equation}
where $E$ is given by \eqref{e:definitione} with $p\cdot n(x)-\rho$ instead of $p\cdot n(x)$.
Let $(t'_j,x'_j)\in[t_j-M\delta_j^{1/2},t_j+M\delta_j^{1/2}]\times B_{x_j,\varepsilon}$ be such that
\begin{align*}
&E_{\varepsilon,\delta_j}(t_j,x_j,u^{\varepsilon,\delta_j}(t_j,x_j),J[u^{\varepsilon,\delta_j}](t_j,x_j)+K_{(0,t_j)}[u^{\varepsilon,\delta_j}](t_j,x_j),p_j,X_j)\\
&=E_*(t'_j,x'_j,u^{\varepsilon,\delta_j}(t_j,x_j),J[u^{\varepsilon,\delta_j}](t_j,x_j)+K_{(0,t_j)}[u^{\varepsilon,\delta_j}](t_j,x_j),p_j,X_j).
\end{align*}
As is well-known, the uniformly exterior sphere condition for $\Omega$ (A5) implies that
$$
p_j\cdot n(x'_j)+\rho\ge0
$$
for large $j$ (see \cite[Section 7.B]{CrandallIshiiLions1992} for example).
Therefore \eqref{e:ishiiboundary3} can be rewritten as
$$
J[u^{\varepsilon,\delta_j}](t_j,x_j)+K_{(0,t_j)}[u^{\varepsilon,\delta_j}](t_j,x_j)+F(t'_j,x'_j,u^{\varepsilon,\delta_j}(t_j,x_j),p_j,X_j)\le0
$$
for all large $j$.
Similarly for a supersolution $v$ of \eqref{e:boundaryishii2}, we get
$$
J[v_{\varepsilon,\delta_j}](s_j,y_j)+K_{(0,s_j)}[v_{\varepsilon,\delta_j}](s_j,y_j)+G(s'_j,y'_j,v_{\varepsilon,\delta_j}(s_j,y_j),-q_j,-Y_j)\le0
$$
for some $(s_j',y_j')\in [s_j-M\delta_j^{1/2},s_j+M\delta_j^{1/2}]\times B_{y_j,\varepsilon}$.
Subtract from the inequality below from the above inequality to obtain
\begin{align*}
&J[u^{\varepsilon,\delta_j}](t_j,x_j)-J[v_{\varepsilon,\delta_j}](s_j,y_j)+K_{(0,t_j)}[u^{\varepsilon,\delta_j}](t_j,x_j)-K_{(0,s_j)}[v_{\varepsilon,\delta_j}](s_j,y_j)\\
&+F(t'_j,x'_j,u^{\varepsilon,\delta_j}(t_j,x_j),p_j,X_j)-G(s'_j,y'_j,v_{\varepsilon,\delta_j}(s_j,y_j),-q_j,-Y_j)\le0.
\end{align*}
The desired inequality \eqref{e:ine} is immediately obtained after taking the limit $j\to\infty$ as in the proof of Proposition \ref{p:weakishii}.
\end{proof}

We are now able to prove a comparison principle.
\begin{theorem}[Comparison principle]\label{t:comparisonneumann}
Assume (A5), (A6), (A7), (A8), and (A9).
Let $u,v:[0,T]\times\overline{\Omega}\to\mathbf{R}$ be a bounded usc subsolution and a bounded lsc supersolution of \eqref{e:neumann} in $(0,T]\times\overline{\Omega}$, respectively.
If $u(0,\cdot)\le v(0,\cdot)$ on $\overline{\Omega}$, then $u\le v$ on $[0,T]\times\overline{\Omega}$.
\end{theorem}

\begin{proof}
According to \cite[Lemma 7.6]{CrandallIshiiLions1992} there is $\phi\in C^2(\overline{\Omega})$ such that $\nabla\phi(x)\cdot n(x)\ge 1$ for $x\in\partial\Omega$ and $\phi\ge0$ on $\overline{\Omega}$.
It is not hard to see that for $\rho>0$, $\bar{u}(t,x):=u(t,x)-\rho\phi$ is a subsolution of
\begin{equation*}
  \begin{cases}
    \partial_t^\alpha\bar{u}+F(t,x,\bar{u},\nabla \bar{u},\nabla^2\bar{u})-\omega_2(\rho\kappa)=0\quad&\text{in $(0,T]\times\Omega$,}\\
    \nabla \bar{u}\cdot n(x)+\rho=0\quad&\text{on $[0,T]\times\partial\Omega$}
  \end{cases}
\end{equation*}
and $\bar{v}(t,x):=v(t,x)+\rho\phi$ is a supersolution of
\begin{equation*}
  \begin{cases}
    \partial_t^\alpha\bar{v}+F(t,x,\bar{v},\nabla\bar{v},\nabla^2\bar{v})+\omega_2(\rho\kappa)=0\quad&\text{in $(0,T]\times\Omega$,}\\
    \nabla\bar{v}\cdot n(x)-\rho=0\quad&\text{on $[0,T]\times\partial\Omega$,}
  \end{cases}
\end{equation*}
where $\kappa=\max_{\overline{\Omega}}(|\nabla\phi|+\|\nabla^2\phi\|)$.

It suffices to prove that $\bar{u}\le\bar{v}$ on $[0,T]\times\overline{\Omega}$ for each $\rho$; letting $\rho\to0$ concludes the proof.
Let $u$ and $v$ denote $\bar{u}$ and $\bar{v}$, respectively.
We suppose by contradiction that $\sup_{[0,T]\times\overline{\Omega}}(u-v)=:\theta>0$.
Notice that $u(0,\cdot)\le v(0,\cdot)$ on $\overline{\Omega}$.

Let $(t_{\sigma,\varepsilon},x_{\sigma,\varepsilon},y_{\sigma,\varepsilon})$ be a maximum point of
$$
\Phi(t,x,y):=u^\varepsilon(t,x)-v_\varepsilon(t,y)-\frac{|x-y|^2}{\sigma}
$$
on $[0,T]\times\overline{\Omega_\varepsilon}\times\overline{\Omega_\varepsilon}$.
We fix $x\in\overline{\Omega}$ arbitrarily.
It is easy to see that
\begin{align*}
u(t_{\sigma,\varepsilon},x)-\varepsilon^{-1}|x_{\sigma,\varepsilon}-x|^2-v_\varepsilon(t_{\sigma,\varepsilon},x_{\sigma,\varepsilon})
&\le \Phi(t_{\sigma,\varepsilon},x_{\sigma,\varepsilon},x_{\sigma,\varepsilon})\\
&\le \Phi(t_{\sigma,\varepsilon},x_{\sigma,\varepsilon},y_{\sigma,\varepsilon})\\
&\le \sup_{[0,T]\times\overline{\Omega}}|u|+\sup_{[0,T]\times\overline{\Omega}}|v|.
\end{align*}
Thus we have 
$$
|x_{\sigma,\varepsilon}-x|^2\le2(\sup_{[0,T]\times\overline{\Omega}}|u|+\sup_{[0,T]\times\overline{\Omega}}|v|)\varepsilon.
$$
This means that $\dist(x_{\sigma,\varepsilon},\overline{\Omega})<C\varepsilon^{1/2}$ and so $x_{\sigma,\varepsilon}\not\in\partial\Omega_\varepsilon$.
In a symmetric way we also see that $y_{\sigma,\varepsilon}\not\in\partial\Omega_\varepsilon$.
Since it turns out that $(t_{\sigma,\varepsilon},x_{\sigma,\varepsilon},y_{\sigma,\varepsilon})$ converges to a $(\hat{t},\hat{x},\hat{x})\in(0,T]\times\overline{\Omega}\times\overline{\Omega}$ such that $(u-v)(\hat{t},\hat{x})=\theta$, we know that $t_{\sigma,\varepsilon}\in(0,T]$ for small $\sigma,\varepsilon$.

Proposition \ref{p:boundaryishii} yields $X,Y\in\mathbf{S}^d$ satisfying
\begin{equation*}
  -\frac{2}{\varepsilon}\left(\begin{array}{cc}I&O\\ O&I\end{array}\right)
  \le\left(
  \begin{array}{cc}
  X & O\\
  O & Y
  \end{array}\right) \le\frac{2}{\sigma}\left(
  \begin{array}{cc}
  I & -I\\
  -I & I
  \end{array}\right)
\end{equation*}
such that
\begin{align*}
&J[u^\varepsilon](t_{\sigma,\varepsilon},x_{\sigma,\varepsilon})-J[v_\varepsilon](t_{\sigma,\varepsilon},y_{\sigma,\varepsilon})+K_{(0,t_{\sigma,\varepsilon})}[u^\varepsilon](t_{\sigma,\varepsilon},x_{\sigma,\varepsilon})-K_{(0,t_{\sigma,\varepsilon})}[v_\varepsilon](t_{\sigma,\varepsilon},y_{\sigma,\varepsilon})\\
&+F_\varepsilon(t_{\sigma,\varepsilon},x_{\sigma,\varepsilon},u^\varepsilon(t_{\sigma,\varepsilon},x_{\sigma,\varepsilon}),p,X)-F^\varepsilon(t_{\sigma,\varepsilon},y_{\sigma,\varepsilon},v_\varepsilon(t_{\sigma,\varepsilon},y_{\sigma,\varepsilon}),p,-Y)\le2\omega_2(\rho\kappa),
\end{align*}
where $p:=2\sigma^{-1}(x_{\sigma,\varepsilon}-y_{\sigma,\varepsilon})$.
As in the proof of Theorem \ref{t:comparison}, letting $\sigma,\varepsilon\to0$ results in
$$
\frac{\theta-(u-v)(0,\hat{x})}{\hat{t}^\alpha\Gamma(1-\alpha)}\le2\omega_2(\rho\kappa),
$$
which is a contradiction for small $\rho$.
\end{proof}

\begin{theorem}[Unique existence]
Assume (A5), (A6), (A7), (A8), (A9), and (A10).
Let $u_-,u_+:[0,T]\times\overline{\Omega}\to\mathbf{R}$ be a subsolution and a supersolution of \eqref{e:neumann}-\eqref{e:initial} with $(u_-)_*>-\infty$, $(u_+)^*<+\infty$ in $[0,T]\times\overline{\Omega}$.
Assume that $(u_-)_*(0,\cdot)=(u_+)^*(0,\cdot)=u_0$ on $\overline{\Omega}$.
Then there exists a unique solution $u\in C([0,T]\times\overline{\Omega})$ of \eqref{e:neumann}-\eqref{e:initial} that satisfies $u_-\le u\le u_+$ in $[0,T]\times\overline{\Omega}$.
\end{theorem}

The proof of the existence part is a trivial modification of that of Theorem \ref{t:existence} with the use of Theorem \ref{t:comparisonneumann}.

As in the case of the Cauchy-Dirichlet problem, sub- and supersolutions can be obtained by slightly extending conventional methods. For example, we set
$$
u_-(t,x)=-M-\frac{C}{\Gamma(1+\alpha)}t^\alpha+d(x)\quad\text{and}\quad u_+(t,x)=-u_-(t,x),
$$
where $M$ and $C$ are sufficiently large constants and $d$ is a function which agrees with the distance to the boundary in a neighborhood of $\partial\Omega$. Then it is easy to check that $u_-$ and $u_+$ are respectively a subsolution and a supersolution of
$$
\partial_t^\alpha u-\Delta u=0
$$
with the same initial-boundary conditions as \eqref{e:neumann}-\eqref{e:initial} if $\Omega$ is smooth; cf. \cite{Barles} for generalization of $F$.

\subsection*{Acknowledgments}
Part of the result in this paper was established while the author was visiting the University of Warsaw during May and June 2017.
Its hospitality is gratefully acknowledged.
The author would like to thank Dr. Yikan Liu for letting him know references on linear equations with Caputo time fractional derivatives.
This work is supported by Grant-in-aid for Scientific Research of JSPS Fellows No. 16J03422 and partly the Program for Leading Graduate Schools, MEXT, Japan.


\providecommand{\bysame}{\leavevmode\hbox to3em{\hrulefill}\thinspace}
\providecommand{\MR}{\relax\ifhmode\unskip\space\fi MR }
\providecommand{\MRhref}[2]{%
  \href{http://www.ams.org/mathscinet-getitem?mr=#1}{#2}
}
\providecommand{\href}[2]{#2}

\end{document}